\documentclass[11pt,a4paper]{amsart}
\usepackage{amsmath, amsthm, amssymb, amsfonts, amscd}
\usepackage[english]{babel}
\usepackage[utf8, latin2]{inputenc}
\usepackage[T1]{fontenc}
\usepackage[mathscr]{eucal}
\usepackage[margin=2cm]{geometry}
\usepackage{graphicx}
\usepackage[all,cmtip]{xy}
\usepackage{tikz-cd}
\usepackage{tkz-graph}
\usepackage{tikz}
    \usetikzlibrary{matrix}
    \usetikzlibrary{arrows}
    \usetikzlibrary{babel}
\usepackage{mathtools}
\usepackage{eqnarray}
\usepackage[colorlinks,linkcolor=black,citecolor=red,urlcolor=red]{hyperref}
\usepackage{enumerate}
 
\theoremstyle{plain}
    \newtheorem{theorem}{Theorem}[section]
    \newtheorem{corollary}[theorem]{Corollary}
    \newtheorem{lemma}[theorem]{Lemma}
    \newtheorem{proposition}[theorem]{Proposition}
    \newtheorem*{claim}{Claim}
\theoremstyle{definition}
    \newtheorem{definition}[theorem]{Definition}
    
\theoremstyle{remark}
    \newtheorem{remark}[theorem]{Remark}

\newcommand{\Ima}{\operatorname{Im}}

\newcommand{\Hom}{\operatorname{Hom}}
\newcommand{\sHom}{\underline{\Hom}}          %Sheaf
\newcommand{\Ext}{\operatorname{Ext}}
\newcommand{\sExt}{\underline{\Ext}}               %Sheaf
\newcommand{\Extrig}{\operatorname{Extrig}}
\newcommand{\sExtrig}{\underline{\Extrig}}      %Sheaf

\newcommand{\CH}{\operatorname{H}}
\newcommand{\DR}{\operatorname{H}_\text{\rm dR}}
\newcommand{\TDR}{\operatorname{T}_\text{\rm dR}}

\newcommand{\Lie}{\operatorname{Lie}}
\newcommand{\Pic}{\operatorname{Pic}}
\newcommand{\Alb}{\operatorname{Alb}}
\newcommand{\Div}{\operatorname{Div}}
\newcommand{\supp}{\operatorname{supp}}
\newcommand{\reg}{\text{\rm reg}}
\newcommand{\us}{\, \cdot \,}

\title{Height pairings of 1-motives}
\author{Carolina Rivera Arredondo}
\date{}
\email{rivera@math.unipd.it}
\address{Dipartimento di Matematica ``Tullio Levi-Civita'', via Trieste 63, I-35121 Padova, Italy}

\begin{document}

\maketitle

\begin{abstract}
The purpose of this work is to generalize, in the context of 1-motives, the $p$-adic height pairings constructed by B. Mazur and J. Tate on abelian varieties. Following their approach, we define a global pairing between the rational points of a 1-motive and its dual. We also provide local pairings between zero-cycles and divisors on a curve, which is done by considering its Picard and Albanese 1-motives.
\end{abstract}

%\tableofcontents

%----------------------------------------------------------------------------------------------------
%----------------------------------------------------------------------------------------------------

\section{Introduction}

In \cite{MT83} Mazur and Tate gave a construction of a global pairing on the rational points of paired abelian varieties over a global field, as well as N\'eron-type local pairings between disjoint zero-cycles and divisors on an abelian variety over a local field. Their approach involved the concept of $\rho$-splittings of biextensions of abelian groups, which they mainly studied in the case of $K$-rational sections of a $\mathbb G_m$-biextension of abelian varieties over a local field. When certain requirements on the base field, the morphism $\rho$, and the abelian varieties are met, they proved the existence of canonical $\rho$-splittings for this type of biextensions, which they later used to construct canonical local pairings between disjoint zero-cycles and divisors on an abelian variety. By considering a global field endowed with a set of places and its respective completions, they were also able to construct a global pairing on the rational points of paired abelian varieties. \\

It will be of particular interest to us, the Poincar\'e biextension of an abelian variety and its dual defined over a non-archimedean local field of characteristic 0. When considering this biextension, there is another method of obtaining $\rho$-splittings, due to Zarhin \cite{ZA90}, starting from splittings of the Hodge filtration of the first de Rham cohomology group of the abelian variety. His construction coincides with Mazur and Tate's in the case that $\rho$ is unramified, or when $\rho$ is ramified and the splitting of the Hodge filtration is the one induced by the unit root subspace. In the latter case, the equality of both constructions is a result of Coleman \cite{CO91}, in the case of ordinary reduction, and of Iovita and Werner \cite{IW03}, in the case of semistable ordinary reduction. \\

For our generalization to 1-motives we will focus on the ramified case. Following Zarhin's approach, we construct $\rho$-splittings of the Poincar\'e biextension of a 1-motive and its dual starting from a pair of splittings of the Hodge filtrations of their de Rham realizations; this is done in Section \ref{sec:ramified}. In order to construct pairings from these $\rho$-splittings, we need them to be compatible with the canonical linearization associated to the biextension; the conditions under which this happens are studied in Section \ref{sec:linearization}. \\

In Section \ref{sec:local_pairing} we consider a semi-normal irreducible curve $C$ over a finite extension of $\mathbb Q_p$ and construct a local pairing between disjoint zero-cycles of degree zero on $C$ and on its regular locus $C_{\reg}$. We do this by considering the Poincar\'e biextension of the Picard and Albanese 1-motives of $C$. This construction generalizes the local pairing of Mazur and Tate \cite[p. 212]{MT83} in the case of elliptic curves. \\

Finally, in Section \ref{sec:global_pairing} we consider a 1-motive $M$ over a number field $F$, a set of places of $F$, and homomorphisms $\rho_v: F_v^* \to \mathbb Q_p$ (almost all vanishing on the units of the valuation ring), with $v$ running through the set of places, as well as a $\rho_v$-splitting $\psi_v$, for each $v$, on the $F_v$-rational sections of the Poincar\'e biextension $P$ of $M$ and its dual $M^\vee$ (satisfying certain properties). With this data we construct a global pairing between the $F$-rational points of $M$ and $M^\vee$ under the condition that, for each ramified $\rho_v$, the $\rho_v$-splitting $\psi_v$ is compatible with the canonical linearization of $P$. The pairing is defined similarly to the case of abelian varieties, hence generalizing the global pairing of Mazur and Tate \cite[Lemma 3.1, p. 214]{MT83} in the case of an abelian variety and its dual.

%\section{Notation}
% Fiber of a biextension/torsor
% Identification of a scheme and the sheaf it represents

%----------------------------------------------------------------------------------------------------

\subsection*{Acknowledgements}
The present article is a product of my PhD thesis which I carried out at the University of Milan and the University of Bordeaux within the framework of the ALGANT-DOC doctoral programme. I am deeply grateful to my advisor at Milan, where I spent most of the time, Luca Barbieri-Viale, for suggesting this topic and introducing me to this very interesting field of mathematics. I would also like to thank my advisors at Bordeaux, Boas Erez and Qing Liu, for their support and guidance.

%----------------------------------------------------------------------------------------------------

\section{Preliminaries on abelian varieties and 1-motives}

\subsection{$\rho$-splittings on abelian varieties}

For the definition of biextension of abelian groups and group schemes we refer to \cite{MU69}.

\begin{definition}[{\cite[p. 199]{MT83}}]
Let $A$, $B$, $H$ , $Y$ be abelian groups and $P$ a biextension of $(A, B)$ by $H$. Let $\rho: H \to Y$ be a homomorphism. A \emph{$\rho$-splitting} of $P$ is a map $\psi: P \to Y$ such that
\begin{enumerate}[(i)]
\item $\psi(h+x)=\rho(h) + \psi(x)$, for all $h \in H$ and $x \in P$ and
\item for each $a \in A$ (resp. $b \in B$) the restriction of $\psi$ to $P_{a, B}$ (resp. $P_{A, b}$) is a group homomorphism,
\end{enumerate}
where $P_{a, B}$ (resp. $P_{A, b}$) denotes the fiber of $P$ over $\{a\} \times B$ (resp. $A \times \{b\}$).
\end{definition} 
\noindent Thus, a $\rho$-splitting can be seen as a bi-homomorphic map which is compatible with the natural actions of $H$. Moreover, $\psi$ induces a trivialization of the pushout of $P$ along $\rho$, hence its name. \\

The context in which these maps were classically studied is the following. Consider a field $K$ which is complete with respect to a place $v$, either archimedean or discrete, $A$ and $B$ abelian varieties over $K$, $P$ a biextension of $(A, B)$ by $\mathbb G_m$, and $\rho: K^* \to Y$ a homomorphism from the group of units of $K$ to an abelian group $Y$. A key result by Mazur and Tate \cite[p. 199]{MT83} states the existence of canonical $\rho$-splittings of the group $P(K)$ of rational points of $P$ in the following cases:
\begin{enumerate}[(i)]
\item $v$ is archimedean and $\rho(c) = 0$ for all $c$ such that $|c|_v = 1$, 
\item $v$ is discrete, $\rho$ is unramified (\textit{i.e.} $\rho(R^*) = 0$, where $R$ is the valuation ring of $K$) and $Y$ is uniquely divisible by $N$, and 
\item $v$ is discrete, the residue field of $K$ is finite, $A$ has semistable ordinary reduction and $Y$ is uniquely divisible by $M$,
\end{enumerate}
where $N$ is an integer depending on $A$ and $M$ is an integer depending on $A$ and $B$. We will mainly focus on case (iii). In this case, the $\rho$-splitting of $P(K)$ is obtained by extending a local formal splitting of $P$, which exists and is unique because of the semistable ordinary reduction of $A$. \\
%, which precisely states the following. Suppose that $v$ is a discrete place with finite residue field $k$ and that $A$ has semistable ordinary reduction. Denote $A_R$ the N\'eron model of $A$, $A_R^0$ its identity component and $m_A$ the exponent of $A_k(k)/A^0_k(k)$. Moreover, denote $T_k$ the maximal torus in $\mathcal A_k$ and $n_A$ the exponent of $A^0_k(k)/T_k(k)$. Define $m_B$ and $n_B$ analogously. Then, if $Y$ is uniquely divisible by $m_A m_B n_A n_B$, there exists a canonical $\rho$-splitting $$\psi_{\rho}: P_K(K) \to Y.$$

When $B = A^\vee$ is the dual abelian variety of $A$ and $P = P_A$ is the Poincar\'e biextension, there is an alternate method of obtaining $\rho$-splittings of $P(K)$ starting with a splitting of the Hodge filtration of the first de Rham cohomology of $A$. This construction is due to Zarhin \cite{ZA90} and is done as follows. Let $K$ be a field which is the completion of a number field with respect to a discrete place $v$ over a prime $p$ and consider a continuous homomorphism $\rho: K^* \rightarrow \mathbb Q_p$. Recall that, associated to the first de Rham cohomology $K$-vector space of $A$, there is a canonical extension 
\begin{equation} \label{hodgefil}
0 \rightarrow \CH^0(A, \Omega_{A/K}^1) \rightarrow \DR^1(A) \rightarrow \CH^1(A, \mathcal O_A) \rightarrow 0
\end{equation}
coming from the Hodge filtration of $\DR^1(A)$.
It is known that \eqref{hodgefil} can be naturally identified with the exact sequence of Lie algebras induced by the universal vectorial extension $A^{\vee \#}$ of $A^\vee$:
\begin{equation}
0 \rightarrow \omega_A \rightarrow A^{\vee \#} \rightarrow A^\vee \rightarrow 0,
\end{equation}
where $\omega_A$ is the $K$-vector group representing the sheaf of invariant differentials on $A$ (see \cite[Prop. 4.1.7, p. 48]{MM74}). Therefore, it is possible to obtain a (uniquely determined) splitting $\eta: A^\vee(K) \rightarrow A^{\vee \#}(K)$ at the level of groups from any splitting $r: \CH^1(A, \mathcal O_A) \rightarrow \DR^1(A)$ of \eqref{hodgefil} (see \cite[Ex. 3.1.5, p. 328]{ZA90} or \cite[Lemma 3.1.1, p. 641]{CO91}). Since $A^{\vee}$ represents the functor $\sExt_K(A, \mathbb G_{m})$, while $A^{\vee \#}$ represents the functor $\sExtrig_K(A, \mathbb G_{m})$ of rigidified extensions of $A$ by $\mathbb G_m$, then the morphism $\eta$ gives a multiplicative way of associating a rigidification to every extension of $A$ by $\mathbb G_{m}$. Indeed, take a point $a^\vee \in A^\vee(K)$ and let $P_{A,a^\vee}$ be the fiber of the Poincar\'e bundle $P_A$ over $A \times \{a^\vee\}$. Then $\eta(a^\vee)$ corresponds to the extension $P_{A,a^\vee}$ of $A$ by $\mathbb G_m$ endowed with a rigidification or, equivalently, a splitting 
$$t_{a^\vee}: \Lie P_{A,a^\vee}(K) \rightarrow \Lie \mathbb G_m(K)$$
of the exact sequence of Lie algebras induced by $P_{A,a^\vee}$. The composition $\Lie \rho \circ t_{a^\vee}$ can then be extended to a group homomorphism $P_{A,a^\vee}(K) \to \mathbb Q_p$ (see \cite[Thm. 3.1.7, p. 329]{ZA90}), for every $a^\vee \in A^\vee$, thus obtaining a $\rho$-splitting 
$$\psi_\rho: P_{A}(K) \rightarrow \mathbb Q_p.$$

When $\rho$ is unramified, $\psi_\rho$ does not depend on the choice of splitting of \eqref{hodgefil}, recovering Mazur and Tate's result for case (ii) (see \cite[Thm. 4.1, p. 331]{ZA90}). On the other hand, when $\rho$ is ramified, $\psi_\rho$ does depend on the chosen splitting of \eqref{hodgefil} (see \cite[Thm. 4.3, p. 333]{ZA90}). Coleman \cite{CO91} demonstrated that, when $A$ has good ordinary reduction, the canonical $\rho$-splitting of $P_A(K)$ constructed by Mazur and Tate comes from the splitting of \eqref{hodgefil} induced by the unit root subspace, which is the subspace of $\DR^1(A)$ on which the Frobenius acts with slope 0. Later, Iovita and Werner \cite{IW03} were able to generalize this result to abelian varieties with semistable ordinary reduction by considering their Raynaud extension, which can be seen as a 1-motive whose abelian part has good ordinary reduction (see also \cite{WE98}).

%----------------------------------------------------------------------------------------------------

\subsection{1-motives}

According to Deligne \cite[p. 59]{DE74}, a \emph{1-motive} $M$ over a field $K$ consists of:
\begin{enumerate}[(i)]
\item a \emph{lattice} $L$ over $K$, \textit{i.e.} a group scheme which, locally for the \'etale topology on $K$, is isomorphic to a finitely generated free abelian constant group;
\item a \emph{semi-abelian variety} $G$ over $K$, \textit{i.e.} an extension of an abelian variety $A$ by a torus $T$; and
\item a morphism of $K$-group schemes $u: L \rightarrow G$.
\end{enumerate}
A 1-motive can be considered as a complex of $K$-group schemes with the lattice in degree -1 and the semi-abelian in degree 0. A \emph{morphism of 1-motives} can then be defined as a morphism of the corresponding complexes.

\subsubsection{Cartier duality}
Associated to a 1-motive $M$ there is a \emph{Cartier dual 1-motive} $M^\vee = [L^\vee \xrightarrow{u^\vee} G^\vee]$ defined as follows (see \cite[p. 67]{DE74}). The lattice $L^\vee := \sHom_K(T, \mathbb G_m)$ is the Cartier dual of $T$, the torus $T^\vee := \sHom_K(L, \mathbb G_m)$ is the Cartier dual of $L$, the abelian variety $A^\vee$ is the dual abelian variety of $A$, and the semi-abelian variety $G^\vee$ is the image of $v: L \xrightarrow{u} G \to A$ under the natural isomorphism
$$\Hom_K(L, A) \xrightarrow{\cong} \Ext_K^1(A^\vee, T^\vee).$$
%For this, we first recall the definition of biextension.
%\begin{definition}
%Let $H, A, B$ be sheaves of abelian groups on a site. A \emph{biextension} of $(A, B)$ by $H$ is an $H_{A \times B}-$torsor $P$ over $A \times B$ which is endowed with a structure of extension of $B_{A}$ by $H_{A}$ and a structure of extension of $A_{B}$ by $H_{B}$, such that both structures are compatible.
%\end{definition}
%\noindent Deligne \cite[\S 10.2.1]{DE74} generalized this definition to complexes as follows.
%\begin{definition}
%Let $C_1 = [A_1 \rightarrow B_1]$ and $C_2 = [A_2 \rightarrow B_2]$ be two complexes of sheaves of abelian groups concentrated in degrees 0 and -1 and $H$ a sheaf of abelian groups.
%A \emph{biextension} of $(C_1, C_2)$ by $H$ consists of a biextension $P$ of $(B_1, B_2)$ by $H$ endowed with
%\begin{enumerate}[(i)]
%\item a trivialization (biadditive section) of the biextension of $(B_1, A_2)$ by $H$ obtained as the pullback of $P$ over $B_1 \times A_2$ and
%\item a trivialization of the biextension of $(A_1, B_2)$ by $H$ obtained as the pullback of $P$ over $A_1 \times B_2$,
%\end{enumerate}
%such that both coincide over $A_1 \times A_2$.
%\end{definition}
There is a canonical biextension $P$ of $(M, M^\vee)$ by $\mathbb G_m$, called the \emph{Poincar\'e biextension}, expressing the duality between $M$ and $M^\vee$. It is defined as the pullback to $G \times G^\vee$ of the Poincar\'e biextension $P_A$ of $(A, A^\vee)$. $P$ is naturally endowed with trivializations over $L \times G^\vee$ and $G \times L^\vee$ that coincide over $L \times L^\vee$, making it a biextension of $(M, M^\vee)$ by $\mathbb G_m$ (see \cite[p. 60]{DE74}). Using the fact that the group scheme $G^\vee$ represents the sheaf $\sExt_K([L \xrightarrow{v} A], \mathbb G_m)$, it is possible to define the map $u^\vee: L^\vee \to G^\vee$ as
\begin{align*}
u^\vee: \sHom_K(T, \mathbb G_m) & \to \sExt_K([L \xrightarrow{v} A], \mathbb G_m) \\
\chi & \mapsto [L \xrightarrow{\xi} P_{A, v^\vee(x^\vee)}],
\end{align*}
where $x^\vee \in L^\vee$ is the element corresponding to $\chi \in \sHom_K(T, \mathbb G_m)$ and $\xi$ is obtained from the trivialization of $P$ over $L \times L^\vee$. \\

\subsubsection{de Rham realization} \label{sec:derham}
%%% Mencionar que diferenciales invariantes corresponden a conexiones.
A 1-motive is endowed with a de Rham realization defined via its universal vectorial extension (see \cite[p. 58]{DE74}). The \emph{universal vectorial extension} of a 1-motive $M = [L \xrightarrow{u} G]$ over $K$ is a two term complex of $K$-group schemes 
$$M^\natural = [L \xrightarrow{u^\natural} G^\natural]$$ 
which is an extension of $M$ by the $K$-vector group $\omega_{G^\vee}$ of invariant differentials on $G^\vee$
\begin{equation} \label{def:UVE}
\xymatrix{
0 \ar[r] & 0 \ar[r] \ar[d] & L \ar@{=}[r] \ar[d]^{u^\natural} & L \ar[r] \ar[d]^{u} & 0 \\
0 \ar[r] & \omega_{G^\vee} \ar[r] & G^\natural \ar[r] & G \ar[r] & 0}
\end{equation}
and satisfies the following universal property: for all $K$-vector groups $V$, the map
$$\Hom_{\mathcal O_K}(\omega_{G^\vee}, V) \rightarrow \Ext^1_K(M, V),$$
%% Decidir notacin para morfismo de grupo vectorial sobre K.
which sends a morphism $\omega_{G^\vee} \to V$ of vector groups to the extension of $M$ by $V$ induced by pushout, is an isomorphism.
It is well known that the universal vectorial extension of a 1-motive always exists. The \emph{de Rham realization} of $M$ is then defined as
\[\TDR(M) = \Lie G^\natural.\]
This is endowed with a \emph{Hodge filtration}, defined as follows:
\[ F^i \TDR(M) = \left\{
\begin{array}{ll}
\TDR(M) & \mbox{if $i \leq -1$,}\\
\omega_{G^\vee} & \mbox{if $i = 0$,}\\
0 & \mbox{if $i \geq 1$.}
\end{array} \right. \]
We mention some properties concerning universal vectorial extensions of subquotients of $M$.

\begin{lemma} \label{lem:UVE}
\begin{enumerate}[(i)]
\item The group scheme $G^\natural$ represents the fppf-sheaf
$$S \mapsto 
\left\{
\begin{array}{c|c}
(g, \nabla) & \text{$g \in G(S)$ and $\nabla$ is a $\natural-$structure on the extension} \\
& \text{$[L^\vee \to P_{g, G^\vee}]$ of $M^\vee$ by $\mathbb G_{m}$ induced by $g$}
\end{array}
\right\}.$$
\item If we regard the semi-abelian variety $G$ as the 1-motive $G[0] = [0 \to G]$, then its universal vectorial extension is a group scheme $G^\#$ which is an extension of $G$ by the vector group $\omega_{A^\vee}$. Moreover, $G^\#$ represents the fppf-sheaf
$$S \mapsto 
\left\{
\begin{array}{c|c}
(g, \nabla) & \text{$g \in G(S)$ and $\nabla$ is a $\natural-$structure on the extension} \\
& \text{of $[L^\vee \xrightarrow{v^\vee} A^\vee]$ by $\mathbb G_{m}$ associated to $g$}
\end{array}
\right\}.$$
\item If we regard the abelian variety $A$ as the 1-motive $A[0] = [0 \to A]$, then its universal vectorial extension is a group scheme $A^\#$ which is an extension of $A$ by the vector group $\omega_{A^\vee}$. Moreover, $A^\#$ represents the fppf-sheaf
$$S \mapsto 
\left\{
\begin{array}{c|c}
(a, \nabla) & \text{$a \in A(S)$ and $\nabla$ is a $\natural-$structure on} \\
& \text{the extension $P_{a, A^\vee}$ of $A^\vee$ by $\mathbb G_{m}$}
\end{array}
\right\}.$$
\item If we regard the lattice $L$ as the 1-motive $L[1] = [L \to 0]$, then its universal vectorial extension is the complex $[L \to \omega_{T^\vee}]$. Via the identifications $L = \sHom_K(T^\vee, \mathbb G_m)$ and $\omega_{T^\vee} = \sHom_{\mathcal O_K}(\Lie T^\vee, \mathcal O_K)$, this map is described as
\begin{align*}
\sHom_K(T^\vee, \mathbb G_m) & \to  \sHom_{\mathcal O_K}(\Lie T^\vee, \mathcal O_K) \\
\chi & \mapsto \Lie \chi .
\end{align*}
\end{enumerate}
\end{lemma}

\begin{proof}
Parts (i) and (ii) follow from Proposition 3.8 and Lemma 5.2 in \cite{BE09}, respectively. Part (iii) follows from Proposition 2.6.7 and Proposition 3.2.3 (a) in \cite{MM74} (see also \cite[Thm. 0.3.1, p. 633]{CO91}). And, finally, (iv) follows from Lemma 2.2.2 in \cite{AB05}, once we notice that there is a natural isomorphism $L \otimes_{\mathbb Z} \mathbb G_a \cong \omega_{T^\vee}$ mapping $x \otimes 1 \mapsto \Lie \chi$.
\end{proof}

Let $P^\natural$ be the biextension of $(M^\natural, M^{\vee \natural})$ by $\mathbb G_m$ obtained from $P$ by pullback. There is a canonical connection $\nabla$ on $P^\natural$ which endows it with a $\natural-$structure (see \cite[Prop. 3.9, p. 1644]{BE09}). Its curvature is an invariant 2-form on $G^\natural \times G^{\vee \natural}$ and therefore it determines an alternating pairing $R$ on $\Lie G^\natural \times \Lie G^{\vee \natural}$ with values in $\Lie \mathbb G_{m}$. Since the restriction of $R$ to $\Lie G^\natural$ and $\Lie G^{\vee \natural}$ is zero, this map induces a pairing
\[\Phi: \Lie G^\natural \times \Lie G^{\vee \natural} \to \Lie \mathbb G_{m}. \]
\emph{Deligne's pairing} is then defined as          
\[(\us, \us)^{Del}_M :=- \Phi: \TDR(M) \times \TDR(M^\vee) \to \Lie \mathbb G_{m}. \]

\subsubsection{Albanese and Picard 1-motives} \label{sec:alb_pic}
Let $C_0$ be a curve over a field $K$ of characteristic 0, \textit{i.e.} a purely 1-dimensional variety \footnote{Originally, Deligne considered only algebraically closed fields, but these constructions can also be done over an arbitrary field of characteristic 0 (see \cite[p. 87--90]{BS01}).}. Consider the following commutative diagram
\[\begin{tikzcd}
C'  \ar[d, twoheadrightarrow, "\pi"'] \ar[r, hook, "j'"] \ar[dd, bend right=50, "q"'] & \bar C' \ar[d, twoheadrightarrow, "\bar \pi"] \\
C \ar[r, hook, "j"] \ar[d, twoheadrightarrow, "\pi_0"'] & \bar C \\
C_0 &
\end{tikzcd}\]
where $C'$ is the normalization of $C_0$, $\bar C'$ is a smooth compactification of $C'$, and $\bar C$ (resp. $C$) is the curve obtained from $\bar C'$ (resp. $C'$) by contracting each of the finite sets $q^{-1}(x)$, for $x \in C_0$. Notice that $\bar C$ is projective and $C$ is semi-normal. Let $S$ be the set of singular points of $C$, $S' := \pi^{-1}(S)$, and $F := \bar C' - C' = \bar C - C$. \\

The \emph{cohomological Albanese 1-motive of $C_0$} is defined as
\[\Alb^+(C_0) = [u_{\mathrm{Alb}}: \Div^0_F(\bar C') \to \Pic^0(\bar C)], \]
where :
\begin{enumerate}[(i)]
\item $\Pic^0(\bar C)$ denotes the group of isomorphism classes of invertible sheaves on $\bar C$ which are algebraically equivalent to 0. This is a semi-abelian variety: the map $\bar \pi^*: \Pic^0(\bar C) \to \Pic^0(\bar C')$ is surjective and its kernel is a torus.
\item $\Div^0_F(\bar C')$ denotes the group of (Cartier) divisors $D$ on $\bar C'$ such that $\supp D \subset F$ and $\mathcal O(D) \in \Pic^0(\bar C')$. 
\item $u_{\mathrm{Alb}}$ is the map $D \mapsto \mathcal O(D)$ associating a divisor $D$ to the corresponding invertible sheaf $\mathcal O(D)$.
\end{enumerate}

The \emph{homological Picard 1-motive of $C_0$} is defined as
\[\Pic^-(C_0) = [u_{\mathrm{Pic}}: \Div^0_{S'/S}(\bar C', F) \to \Pic^0(\bar C', F)], \]
where :
\begin{enumerate}[(i)]
\item $\Pic^0(\bar C', F)$ denotes the group of isomorphism classes of pairs $(\mathcal L, \phi)$, where $\mathcal L$ is an invertible sheaf on $\bar C'$ algebraically equivalent to 0 and $\phi: \mathcal L|_{F} \to \mathcal O_F$ is a trivialization over $F$. This is a semi-abelian variety: the natural map $\Pic^0(\bar C', F) \to \Pic^0(\bar C')$ is surjective and its kernel is a torus.
\item $\Div^0_{S'/S}(\bar C', F)$ denotes the group of (Cartier) divisors $D$ on $\bar C'$ which belong to the kernel of $\bar \pi_*: \Div_{S'}(\bar C') \to \Div_S(\bar C)$ and satisfy that $\mathcal O(D) \in \Pic^0(\bar C', F)$.
\item $u_{\mathrm{Pic}}$ is the map $D \mapsto \mathcal O(D)$ associating a divisor $D$ to the corresponding invertible sheaf $\mathcal O(D)$.
\end{enumerate}

An important fact is that the dual of $\Pic^-(C_0)$ is $\Alb^+(C_0)$ and viceversa.

%----------------------------------------------------------------------------------------------------

\section{Linearizations of biextensions} \label{sec:linearization}

%Set notation for Poincare biextension. #writethis
%Let $P_K$ be the Poincar\'e biextension of $(M_K, M_K^\vee)$ by $\mathbb G_{m,K}$ with trivializations $\tau_L$ and $\tau_{L^\vee}$ over $L_K \times G_K^\vee$ and $G_K \times L_K^\vee$, respectively, which coincide over $L_K \times L_K^\vee$
%\[\xymatrix{
 %& P_K \ar[d] & \\
%L_K \times G_K^\vee \ar[r]_{u_K \times Id} \ar[ur]^{\tau_L} & G_K \times G_K^\vee & G_K \times L_K^\vee \ar[ul]_{\tau_{L^\vee}} \ar[l]^{Id \times u_K^\vee}
%\rlap{\ .}}\]
%We have that $P_K(K)$ is a biextension of $(G_K(K), G_K^\vee(K))$ by $K^*$ which has compatible trivializations over $L_K(K) \times G_K^\vee(K)$ and $G_K(K) \times L_K^\vee(K)$. 
%We also denote $P_{A_K}$ the Poincar\'e biextension of $(A_K, A_K^\vee)$ by $\mathbb G_{m,K}$. \\

In this section, we consider commutative group schemes over a field $K$. We give the following definition.

\begin{definition} \label{def:linearization}
Let  $C = [A \xrightarrow{u} B], C' = [A' \xrightarrow{u'} B']$ be complexes of commutative group schemes over $K$. Let
\begin{align*}
\sigma: A \times B & \to B \\
(a, b) & \mapsto u(a) + b
\end{align*}
be the $A$-action on $B$ induced by $u$, and define $\sigma': A'  \times B'$ analogously. Let $P$ be a biextension of $(B, B')$ by $\mathbb G_m$. We define an \emph{$A \times A'$-linearization} of $P$ as an $A \times A'$-action on $P$
\[\Sigma: (A \times A') \times P \to P\]
satisfying the following conditions:
\begin{enumerate}[(i)]
    \item \emph{$\mathbb G_{m}$-equivariance}: For $a \in A$, $a' \in A'$, $c \in \mathbb G_{m}$ and $x \in P$,
    \[\Sigma(a, a', c + x) = c + \Sigma(a, a', x).\]
    \item \emph{Compatibility with $\sigma$ and $\sigma'$}: For $a \in A$ and $a' \in A'$, if $x \in P$ lies above $(b, b') \in B \times B'$ then $\Sigma(a, a', x)$ lies above $(\sigma(a, b), \sigma'(a', b'))$.
    \item \emph{Compatibility with the partial group structures of $P$}: For $a \in A$, $a'_1, a'_2 \in A'$ and $x_1, x_2 \in P$ lying above $b \in B$,
    \[\Sigma(a, a_1' + a_2', x_1 +_1 x_2) = \Sigma(a, a'_1, x_1) +_1 \Sigma(a, a'_2, x_2),\]
    and for $a_1, a_2 \in A$, $a' \in A'$ and $x_1, x_2 \in P$ lying above $b' \in B'$,
    \[\Sigma(a_1 + a_2, a', x_1 +_2 x_2) = \Sigma(a_1, a', x_1) +_2 \Sigma(a_2, a', x_2).\]
\end{enumerate}
\end{definition}

\begin{remark}
An action $\Sigma: (A \times A') \times P \to P$ satisfying conditions (i) and (ii) is an $A \times A'$-linearization of the line bundle $P$ in the sense of Definition 1.6 in \cite[p. 30]{MF94}; this can be summed up as saying that $\Sigma$ is a ``bundle action'' lifting the actions $\sigma$ and $\sigma'$. Notice that $\sigma$ and $\sigma'$ are homomorphisms, and so condition (iii) may then be interpreted as a lifting to $P$ of the compatibility of $\sigma$ and $\sigma'$ with the group structures of $B$ and $B'$. In the rest of the article, we will only use the term \textit{linearization} in the sense of Definition \ref{def:linearization} above.
%See also "Lectures on Invariant Theory" by Igor Dolgachev, p. 104 for definition of linearization.
\end{remark}

\begin{remark}
By considering constant group schemes, we are also able to talk about linearizations of biextensions of abelian groups.
\end{remark}

Let $C = [A \xrightarrow{u} B], C' = [A' \xrightarrow{u'} B']$ be as in Definition \ref{def:linearization} and consider a biextension $P$ of $(B, B')$ by $\mathbb G_{m}$. Given a biextension structure of $(C, C')$ by $\mathbb G_{m}$ on $P$ with trivializations
$$\tau: A \times B' \to P, \quad \tau': B \times A' \to P$$
we can define an $A \times A'$-linearization of $P$ as
\begin{align*}
\Sigma: (A \times A') \times P & \to P \\
(a,a',x) & \mapsto [\tau'(u(a), a') +_2 \tau'(b, a')] +_1 [\tau(a, b') +_2 x],
\end{align*}
where $x \in P$ lies above $(b, b') \in B \times B'$.
This construction is due to \cite[Thm. 6.8, p. 688]{BL91} (see also \cite[p. 306]{WE98}).
Conversely, given an $A \times A'$-linearization
\[\Sigma: (A \times A') \times P \to P\]
of $P$, we can define a biextension structure of $(C, C')$ by $\mathbb G_{m}$ on $P$ as the one determined by the trivializations
\begin{align*}
    \tau: A \times B' & \to P \\
    (a, b') & \mapsto \Sigma(a, 0, 0_{b'})
\end{align*}
\begin{align*}
    \tau': B \times A' & \to P \\
    (b, a') & \mapsto \Sigma(0, a', 0_b),
\end{align*}
where $0_b, 0_{b'}$ are the zero elements in the groups $(P_{b, B'}, +_1), (P_{B, b'}, +_2)$, respectively. These constructions are inverses of each other. \\

\begin{proposition} \label{pro:quotient_biext}
Let $C, C'$ an $P$ be as in Definition \ref{def:linearization} and suppose that $u(K)$ and $u'(K)$ are injective. Then an $A \times A'$-linearization $\Sigma$ of $P$ induces a quotient biextension $Q(K)$ of $(B(K)/A(K), B'(K)/A'(K))$ by $K^*$. 
\end{proposition}

\begin{proof}
Notice that $P(K)$ is a biextension of $(B(K), B'(K))$ by $K^*$ and that $\Sigma(K): (A(K)\times A'(K)) \times P(K) \to P(K)$ is an $A(K) \times A'(K)$-linearization of $P(K)$. We define $Q(K)$ as the set consisting of the orbits 
\[[x] := \{\Sigma(a, a', x) | a \in A(K), a' \in A'(K)\}\]
of elements $x \in P(K)$ under $\Sigma$. Then $Q(K)$ maps surjectively onto $B(K)/A(K) \times B'(K)/A'(K)$ and is endowed with a $K^*$-action which is free and transitive on fibers. To see that it is a biextension it is then enough to prove that $+_1$ and $+_2$ induce partial group structures on $Q(K)$. For this, take elements $x_1, x_2 \in P(K)$ lying above $(b_1, b_1'), (b_2, b_2') \in B(K) \times B'(K)$, respectively, such that the orbits of $b_1$ and $b_2$ under $\sigma$ are equal. This is equivalent to having
$$b_1 = \sigma(a, b_2),$$
for some (unique) $a \in A(K)$. Then $x_1$ and $\Sigma(a, 0, x_2)$ project to $b_1 \in B(K)$ and we are able to define
\[[x_1] +_1 [x_2] := [x_1 +_1 \Sigma(a, 0, x_2)].\]
This is well defined and commutative. We define the partial group structure $+_2$ analogously. \\
\end{proof}

Consider a pair of 1-motives $M = [L \xrightarrow{u} G]$, $M' = [L' \xrightarrow{u'} G']$ and a biextension $P$ of $(M, M')$ by $\mathbb G_m$. For our purposes, we give the following 
\begin{definition}
The group of \emph{$K$-points of} $M$ over $K$ as 
$$M(K) := \Ext^1_K(M^\vee, \mathbb G_{m}).$$
\end{definition}
\noindent This is inspired by \cite[p. 326]{DE79}. Consider the short exact sequence of complexes
\[\xymatrix{
0 \ar[r] & 0 \ar[r] \ar[d] & L^\vee \ar@{=}[r] \ar[d]^{u^\vee} & L^\vee \ar[r] \ar[d]^{v^\vee} & 0 \\
0 \ar[r] & T^\vee \ar[r] & G^\vee \ar[r] & A^\vee \ar[r] & 0
}\]
and the long exact sequence of abelian groups that it induces
\[\ldots \to L(K) \xrightarrow{u(K)} G(K) \to M(K) \to \Ext_K^1(T^\vee, \mathbb G_{m}) \to \ldots.\]
It follow that, when $T^\vee$ is split (or, equivalently, when $L$ is constant), the group of $K$-points of M is
\begin{equation}
M(K) = G(K)/\Ima(u(K)).
\end{equation}

If $L$, $L'$ are constant and $u(K)$, $u'(K)$ are injective then $P(K)$ induces a biextension of $(M(K), M'(K))$ by $K^*$, by Proposition \ref{pro:quotient_biext}. When $M' = M^\vee$ and $P$ is the Poincar\'e biextension, we will denote by $Q_M(K)$ the induced biextension of $(M(K), M^\vee(K))$ by $K^*$. \\

We will now introduce the concept of \emph{compatibility} between a linearization and a $\rho$-splitting of a biextension. First, we recall the following definition from \cite[p. 199]{MT83}
\begin{definition}
Let $B$, $B'$, $H$ , $Y$ be abelian groups and $P$ a biextension of $(B, B')$ by $H$. Let $\rho: H \to Y$ be a homomorphism. A \emph{$\rho$-splitting} of $P$ is a map $\psi: P \to Y$ such that
\begin{enumerate}[(i)]
\item $\psi(h+x)=\rho(h) + \psi(x)$, for all $h \in H$ and $x \in P$ and
\item for each $b \in B$ (resp. $b' \in B'$) the restriction of $\psi$ to $P_{b, B'}$ (resp. $P_{B, b'}$) is a group homomorphism.
\end{enumerate}
%where $P_{a, B}$ (resp. $P_{A, b}$) denotes the fiber of $P$ over $\{a\} \times B$ (resp. $A \times \{b\}$).
\end{definition} 

\begin{definition}
Let $C = [A \xrightarrow{u} B], C' = [A' \xrightarrow{u'} B']$ be complexes of commutative group schemes over $K$ and $P$ a biextension of $(C, C')$ by $\mathbb G_{m}$. Let $Y$ be an abelian group and $\rho: K^* \to Y$ a homomorphism. We will say that a $\rho$-splitting $\psi: P(K) \to Y$ of $P(K)$ is \emph{compatible} with the induced $A \times A'$-linearization $\Sigma$ of $P$ if any of the following equivalent conditions are satisfied:
\begin{enumerate}[(i)]
    \item $\psi(\Sigma(a, a', x)) = \psi(x)$,
    for all $a \in A(K)$, $a' \in A'(K)$ and $x \in P(K)$,
    \item $\psi \circ \tau$ and $\psi \circ \tau'$ vanish on $A(K) \times B'(K)$ and $B(K) \times A'(K)$, respectively.
\end{enumerate}
\end{definition}

\begin{remark}
Assuming $u(K)$ and $u'(K)$ injective, $\psi$ is compatible with an $A \times A'$-linearization if and only if it induces a $\rho$-splitting on the quotient biextension $Q(K)$, which exists by Proposition \ref{pro:quotient_biext}. \\
\end{remark}

%----------------------------------------------------------------------------------------------------

\section{$\rho$-splittings in the ramified case} \label{sec:ramified}

Let $K$ be a finite extension of $\mathbb Q_p$ and consider a branch $\lambda: K^* \to K$ of the $p$-adic logarithm. For a commutative algebraic group $H$ over $K$, we will denote by $\lambda_H: H(K) \to \Lie H(K)$ the uniquely determined homomorphism of Lie groups extending $\lambda$ as constructed in \cite{ZA96}. Let $M = [L \xrightarrow{u} G]$ be a 1-motive over $K$ with $L$ and $T$ split, and denote $M^\vee = [L^\vee \xrightarrow{u^\vee} G^\vee]$ its dual; notice that $L^\vee$ and $T^\vee$ are also split. Let $M^\natural = [L \xrightarrow{u^\natural} G^\natural]$ and $M^{\vee \natural} = [L \xrightarrow{u^{\vee \natural}} G^{\vee \natural}]$ be their corresponding universal vectorial extensions. The group schemes described in Lemma \ref{lem:UVE} fit in the following commutative diagrams with exact rows and columns: \\
\noindent\begin{minipage}{0.55\linewidth}
\begin{equation} \label{notation_UVE1}
\begin{tikzcd}
& 0 \ar[d] & 0 \ar[d] & & \\
0 \ar[r] & \omega_{A^\vee} \ar[d] \ar[r]  & G^\# \ar[d, "\gamma"'] \ar[r, "\theta'"] & G \ar[d, equal] \ar[r] & 0 \\
0 \ar[r] & \omega_{G^\vee} \ar[d, "\varepsilon"'] \ar[r, "\zeta"] & G^\natural  \arrow[ul, phantom, "\ulcorner", very near start] \ar[d, "\sigma"'] \ar[r, "\theta"] & G \ar[r] & 0 \\
& \omega_{T^\vee} \ar[r, equal] \ar[d] & \omega_{T^\vee} \ar[d] & & \\
& 0 & 0 & & 
\end{tikzcd}
\end{equation}
\end{minipage}%
\begin{minipage}{0.5\linewidth}
\begin{equation} \label{notation_UVE2}
\begin{tikzcd}
& & 0 \ar[d] & 0 \ar[d] & \\ & & T \ar[d, "\iota^\#"'] \ar[r, equal] & T \ar[d, "\iota"] & \\
0 \ar[r] & \omega_{A^\vee} \ar[r] \ar[d, equal] & G^\# \ar[r, "\theta'"] \ar[d, "\pi^\#"'] \arrow[dr, phantom, "\lrcorner", very near start] & G \ar[r] \ar[d, "\pi"] & 0 \\
0 \ar[r] & \omega_{A^\vee} \ar[r] & A^\# \ar[r, "\theta_A"] \ar[d] & A \ar[r] \ar[d] & 0 \\
& & 0 & 0 & 
\rlap{\ .}
\end{tikzcd}
\end{equation}
\end{minipage}

We will denote the morphisms in the diagrams for $G^{\vee}$ analogously, so that $\varepsilon$ is defined by pullback along $\iota^\vee:T^\vee \to G^\vee$ and $\Lie \iota^\vee$ is dual to $\varepsilon$. \\

For the rest of this section, we fix splittings of the following exact sequences of vector group schemes over $K$:
\begin{equation} \label{es_vg1}
\begin{tikzcd}
0 \ar[r] & \omega_{A^\vee} \ar[r] & \omega_{G^\vee} \ar[r, "\varepsilon"'] & \omega_{T^\vee} \ar[r] \ar[l, dashed, bend right, "\bar \varepsilon"'] & 0
\end{tikzcd}
\end{equation}
\[\begin{tikzcd}
0 \ar[r] & \omega_A \ar[r] & \omega_G \ar[r, "\varepsilon^\vee"'] & \omega_T \ar[r] \ar[l, dashed, bend right, "\bar \varepsilon^\vee"'] & 0
\rlap{\ .}
\end{tikzcd}\]

These induce the isomorphisms:
\begin{enumerate}[(i)]
\item $\omega_G \cong \omega_A \times \omega_T$ of vector group schemes, and similarly for $\omega_{G^\vee}$.
%$\omega_{G^\vee} \cong \omega_{A^\vee} \times \omega_{T^\vee}$
\item $G^\natural \cong \omega_{T^\vee} \times G^\#$ of commutative group schemes induced by the section $\bar \sigma := \zeta \circ \bar \varepsilon$ of $\sigma$, and similarly for $G^{\vee \natural}$. We will denote by $\bar \gamma$ the induced retraction of $\gamma$:
\begin{equation} \label{es_Gnat}
\begin{tikzcd}
0 \ar[r] & G^\# \ar[r, "\gamma"'] & G^\natural \ar[r, "\sigma"'] \ar[l, dashed, bend right, "\bar \gamma"'] & \omega_{T^\vee} \ar[r] \ar[l, dashed, bend right, "\bar \sigma"'] & 0
\rlap{\ .}
\end{tikzcd}
\end{equation}
Notice that $\bar \gamma$ satisfies $\theta' \circ \bar \gamma = \theta$, by the universal property of the pushout. We fix the analogous notation for $G^{\vee \natural}$.
%$G^{\vee \natural} \cong \omega_T \times G^{\vee \#}$
\item $\Lie G \cong \Lie A \times \Lie T$ of Lie algebras obtained from (i) by duality. We denote $j := \Lie \iota$, $q := \Lie \pi$ and let $\bar j$ be the retraction of $j$ and $\bar q$ the section of $q$ induced by this isomophism:
\begin{equation} \label{es_la}
\begin{tikzcd}
0 \ar[r] & \Lie T \ar[r, "j"'] & \Lie G \ar[r, "q"'] \ar[l, dashed, bend right, "\bar j"'] & \Lie A \ar[r] \ar[l, dashed, bend right, "\bar q"'] & 0
\rlap{\ .}
\end{tikzcd}
\end{equation}
We also fix the analogous notation for $G^{\vee \natural}$. \\
%$\Lie G^\vee \cong \Lie A^\vee \times \Lie T^\vee$
\end{enumerate}
%Congruence on spacing between pairs of exact sequences written in tikzcd and positioning of labels (consider the cases: labels on the upper arrows, the lower ones and on both). #checkthis

We will continue to denote Deligne's pairing associated to $M$ and its dual as
\[(\us, \us)^{Del}_M: \TDR(M) \times \TDR(M^\vee) = \Lie G^\natural \times \Lie G^{\vee \natural} \to \mathbb G_a .\]
Deligne's pairing associated to $A$ and its dual will be denoted as
\[(\us, \us)^{Del}_A: \TDR(A) \times \TDR(A^\vee) = \Lie A^\# \times \Lie A^{\vee \#} \to \mathbb G_a. \]
We want to recognize in $(\us, \us)^{Del}_M$ the contribution of the abelian varieties and the tori. With this in mind, we also define the following pairing.
\begin{definition}
Define $T^\natural:= \omega_{T^\vee} \times T$ and $T^{\vee \natural}:= \omega_{T} \times T^\vee$. Let $\alpha_{T^\vee}$ be the invariant differential of $T^\vee$ over $\omega_{T^\vee}$ which corresponds to the identity map on $\omega_{T^\vee}$, and define $\alpha_T$ analogously. Denote by $\Phi_T$ the pairing on $\Lie T^\natural \times \Lie T^{\vee \natural}$ determined by the curvature of the invariant differential $\alpha_{T^\vee} + \alpha_T$. We define
\[(\us, \us)_T:= - \Phi_T: \Lie T^\natural \times \Lie T^{\vee \natural} \to \mathbb G_a. \]
\end{definition}

\noindent The following lemma gives an explicit description of $(\us, \us)_T$. 

\begin{lemma} \label{lem:Del_pairing_T}
Let $L \cong \mathbb Z^r$ and $T \cong \mathbb G_m^d$, so that $L^\vee \cong \mathbb Z^d$ and $T^\vee \cong \mathbb G_m^r$. Then the pairing 
\[(\us, \us)_T: \Lie T^\natural \times \Lie T^{\vee \natural} \cong (\mathbb G_a^r \times \mathbb G_a^d) \times (\mathbb G_a^d \times \mathbb G_a^r) \to \mathbb G_a \]
is given by the matrix
\[\Gamma = \phantom{(17') \left\{\vphantom{\dfrac{n}{\eta \cdot \delta}}\right.\hspace{\dimexpr\arraycolsep-\nulldelimiterspace}}% For correct horizontal spacing within display math \[ ... \]
\bordermatrix{\hspace{-\arraycolsep} & 
\overbrace{\hphantom{\hspace{1.5cm}}}^{d} & 
\overbrace{\hphantom{\hspace{2.3cm}}}^{r} 
\cr
\hspace{-\arraycolsep}\mathllap{r \left\{\vphantom{\begin{matrix} -1 & & -1 \\ & \ddots& \\ -1 & & -1 \end{matrix}}\right.\kern-\nulldelimiterspace} & 
\begin{matrix} 0 & & 0 \\ & \ddots& \\ 0 & & 0 \end{matrix}
&
\begin{matrix} -1 & & \quad 0 \\ & \ddots& \\ \quad 0 & & -1 \end{matrix}   
\cr
\hspace{-\arraycolsep}\mathllap{d \left\{\vphantom{\begin{matrix} -1 & & -1 \\ & \ddots& \\ -1 & & -1 \end{matrix}}\right.\kern-\nulldelimiterspace} & 
\begin{matrix} 1 & & 0 \\ & \ddots& \\ 0 & & 1 \end{matrix}
&
\begin{matrix} \quad 0 & & \quad 0 \\ & \ddots& \\ \quad 0 & & \quad 0 \end{matrix}}.\]
\end{lemma}

\begin{proof}
In this case, the global differential $\alpha_{T^\vee} + \alpha_T$ on $T^\natural \times T^{\vee \natural} = (\mathbb G_a^r \times \mathbb G_m^d) \times (\mathbb G_a^d \times \mathbb G_m^r)$ has the expression
\[\alpha_{T^\vee} + \alpha_T = \sum_{i=1}^r x_i \frac{dt_i}{t_i} + \sum_{j=1}^d y_j \frac{dz_j}{z_j}, \]
where $x_i$ (resp. $y_j$) are the parameters of $\mathbb G_a^r$ (resp. $\mathbb G_a^d$) and $t_i$ (resp. $z_j$) are the parameters of $\mathbb G_m^r$ (resp. $\mathbb G_m^d$) (see \cite[Ex. 4.4, p. 1647]{BE09}), and its curvature is
\begin{align*}
d(\alpha_{T^\vee} + \alpha_T) & = \sum_{i=1}^r d x_i \wedge \frac{dt_i}{t_i} + \sum_{j=1}^d d y_j \wedge \frac{dz_j}{z_j} \\
& = \sum_{i=1}^r d x_i \wedge \frac{dt_i}{t_i} - \sum_{j=1}^d \frac{dz_j}{z_j} \wedge d y_j \, .
\end{align*}
From this, it is straightforward that $(\us, \us)_T$ is given by the matrix $\Gamma$.
%(see \cite[\S 10.2.7.3]{DE74}) 
\end{proof}

\begin{definition} \label{def:jnat+qnat}
Define
\begin{align*}
\iota^\natural & := Id \times \iota^\#: T^\natural = \omega_{T^\vee} \times T \to \omega_{T^\vee} \times G^\# \cong G^\natural, \\
\pi^\natural & := \pi^\# \circ \bar \gamma: G^\natural \to A^\#,
\end{align*}
and denote $j^\natural := \Lie \iota^\natural$ and $q^\natural := \Lie \pi^\natural$. Define $\iota^{\vee \natural}, \pi^{\vee \natural}, j^{\vee \natural}, q^{\vee \natural}$ analogously. 
\end{definition}

Notice that the following diagram commutes and the upper and lower rows are exact, which makes the middle row exact as well:
\[\begin{tikzcd}
& 0 \ar[d] & 0 \ar[d] & & \\
0 \ar[r] & T \ar[r, "{\iota^\#}"] \ar[d] & G^\# \ar[r, "{\pi^\#}"] \ar[d, "\gamma"] & A^\# \ar[r] \ar[d, equal] & 0 \\
0 \ar[r] & T^\natural \ar[r, "{\iota^\natural}"] \ar[d] & G^\natural \ar[r, "{\pi^\natural}"] \ar[d, "\sigma"] & A^\# \ar[r] & 0 \\
& \omega_{T^\vee} \ar[r, equal] \ar[d] & \omega_{T^\vee} \ar[d] & & \\
& 0 & 0 & & 
\rlap{\, .}
\end{tikzcd}\]
Therefore, $j^\natural$ and $q^\natural$ fit in a short exact sequence of Lie algebras 
\begin{equation} \label{es_lanat}
\begin{tikzcd}
0 \ar[r] & \Lie T^\natural \ar[r, "j^\natural"'] & \Lie G^\natural \ar[r, "q^\natural"'] \ar[l, dashed, bend right, "\bar j^\natural"'] & \ar[r] \Lie A^\# \ar[r] & 0
\end{tikzcd}
\end{equation}
which has a splitting $\bar j^{\natural}$ induced by $\bar j$ (see diagram \eqref{es_la}). More precisely, $\bar j^\natural$ is given by
$$\bar j^\natural := Id \times (\bar j \circ \Lie \theta'): \Lie G^\natural \cong \omega_{T^\vee} \times \Lie G^\# \to \omega_{T^\vee} \times \Lie T = \Lie T^\natural,$$
and similarly for $\bar j^{\vee \natural}$. Indeed, $\bar j^\natural$ is a splitting of \eqref{es_lanat}: 
$$\bar j^\natural \circ j^\natural = (\bar j \circ \Lie \theta') \circ j^\# = \bar j \circ j = Id.$$ 
Consider the morphisms
\[\Lie T^\natural \times \Lie T^{\vee \natural} \xleftarrow{\bar j^\natural \times \bar j^{\vee \natural}} \Lie G^\natural \times \Lie G^{\vee \natural} \xrightarrow{q^\natural \times q^{\vee \natural}} \Lie A^\# \times \Lie A^{\vee \#} .\]
We have the following
      
\begin{lemma} \label{lem:Del_pairing_M}
For all $(h, h^\vee) \in \Lie G^\natural \times \Lie G^{\vee \natural}$, the following equality holds
\[(h, h^\vee)^{Del}_M = (\bar j^\natural(h), \bar j^{\vee \natural}(h^\vee))_T + (q^\natural(h), q^{\vee \natural}(h^\vee))^{Del}_A .\]
\end{lemma}

\begin{proof}
Recall that $P^\natural$ is defined as the pullback of the Poincar\'e biextension $P$ along $\theta \times \theta^\vee: G^\natural \times G^{\vee \natural} \to G \times G^\vee$, and that $\nabla$ is determined by the sum of two differentials associated to the identities of $G^\natural$ and $G^{\vee \natural}$ (see \cite[Prop.3.9, p. 1644]{BE09}). \\

We will first describe the decomposition of the structure of $\natural$-extension over $G^\natural$ of $P^\natural$ induced by $Id \in G^\natural(G^\natural)$. The split exact sequence
\[\begin{tikzcd}[row sep=scriptsize]
0 \ar[r] & G^\# \ar[r, "\gamma"'] & G^\natural \ar[r, "\sigma"'] \ar[l, bend right, dashed, "\bar \gamma"'] & \omega_{T^\vee} \ar[r] \ar[l, bend right, dashed, "\bar \sigma"'] & 0
\end{tikzcd}\]
induces an isomorphism
\begin{align*}
G^\natural(G^\natural) & \cong \omega_{T^\vee}(G^\natural) \oplus G^\#(G^\natural) \\
Id & \mapsto (\sigma, \bar \gamma) \nonumber .
\end{align*}
By Definition \ref{def:jnat+qnat} we have $\pi^\natural = \pi^\# \circ \bar \gamma$, and so $\bar \gamma \in G^\#(G^\natural)$ and $Id \in A^\#(A^\#)$ map to the same element $\pi^\natural \in A^\#(G^\natural)$ in the diagram below:
\[\begin{tikzcd}[row sep=tiny, column sep=tiny]
G^\#(G^\natural) \ar[r, "\pi^\# \circ \_"] & A^\#(G^\natural) & A^\#(A^\#) \ar[l, "\_ \circ \pi^\natural"'] \\
\bar \gamma \ar[d, "="{sloped}, phantom] \ar[r, mapsto] & \pi^\# \circ \bar \gamma = \pi^\natural \ar[d, "="{sloped}, phantom] & Id \ar[l, mapsto] \ar[d, "="{sloped}, phantom] \\
([L^\vee_{G^\natural} \to (\pi^\natural \times Id)^*P_{A^\# \times A^\vee}], (\pi^\natural \times Id)^*\nabla_{A, 2}) & ((\pi^\natural \times Id)^*P_{A^\# \times A^\vee}, (\pi^\natural \times Id)^*\nabla_{A, 2}) & (P_{A^\# \times A^\vee}, \nabla_{A, 2})
\rlap{\ .}
\end{tikzcd}\]
Hence, if $(P_{A^\# \times A^\vee}, \nabla_{A, 2})$ is the $\natural$-extension of $A^\vee_{A^\#}$ by $\mathbb G_{m, A^\#}$ corresponding to $Id \in A^\#(A^\#)$, by Lemma \ref{lem:UVE} (iii), then $\bar \gamma$ corresponds to $([L^\vee_{G^\natural} \to (\pi^\natural \times Id)^*P_{A^\# \times A^\vee}], (\pi^\natural \times Id)^*\nabla_{A, 2})$. \\

On the other hand, the contribution of $\sigma \in \omega_{T^\vee}(G^\natural)$ is described by the trivial extension of $G^\vee_{G^\natural}$ by $\mathbb G_{m, G^\natural}$ endowed with the connection induced by the invariant differential $\bar \varepsilon \circ \sigma \in \omega_{G^\vee}(G^\natural)$ (see diagram \eqref{es_vg1} for notation). Notice that the invariant differential of $T^\vee$ over $G^\natural$ corresponding to $\sigma \in \omega_{T^\vee}(G^\natural)$ is just the pullback of $\alpha_{T^\vee}$ along $\sigma$. Now, if we consider invariant differentials as morphisms of vector groups, then $\bar \varepsilon \circ \sigma \in \omega_{G^\vee}(G^\natural)$ will correspond to $(\sigma^*\alpha_{T^\vee}) \circ \bar j^\vee$, since we had defined $\bar j^\vee$ as the morphism induced by $\bar \varepsilon$ by duality (see diagram \eqref{es_la} for notation):
%\[\begin{tikzcd}[row sep=tiny]
%\omega_{T^\vee}(\omega_{T^\vee}) \ar[r, "\_ \circ \sigma"] & \omega_{T^\vee}(G^\natural) \ar[r, "\bar \varepsilon^\vee \circ \_"] & \omega_{G^\vee}(G^\natural) \\
%Id \ar[d, equal] \ar[r, mapsto] & \sigma \ar[d, equal] \ar[r, mapsto] & \bar \varepsilon^\vee \circ \sigma \ar[d, equal] \\
%\alpha_{T^\vee} & \sigma^*\alpha_{T^\vee} & (P_{A^\# \times A^\vee}, \nabla_{A, 2})
%\rlap{\ .}
%\end{tikzcd}\]
\[\begin{tikzcd}[row sep=tiny]
\omega_{T^\vee}(\omega_{T^\vee}) \ar[r, "\_ \circ \sigma"]  \ar[d, "="{sloped}, phantom] & \omega_{T^\vee}(G^\natural) \ar[r, "\bar \varepsilon \circ \_"]  \ar[d, "="{sloped}, phantom] & \omega_{G^\vee}(G^\natural)  \ar[d, "="{sloped}, phantom] \\
\sHom_{\mathcal O_{\omega_{T^\vee}}}(\Lie T^\vee_{\omega_{T^\vee}}, \mathbb G_{a, \omega_{T^\vee}}) \ar[r, "\sigma^*"] & \sHom_{\mathcal O_{G^\natural}}(\Lie T^\vee_{G^\natural}, \mathbb G_{a, G^\natural}) \ar[r, "\_ \circ \bar j^\vee"] & \sHom_{\mathcal O_{G^\natural}}(\Lie G^\vee_{G^\natural}, \mathbb G_{a, G^\natural}) \\
\alpha_{T^\vee} \ar[r, mapsto] & \sigma^*\alpha_{T^\vee} \ar[r, mapsto] & (\sigma^*\alpha_{T^\vee}) \circ \bar j^\vee
\rlap{\ .}
\end{tikzcd}\]
%Therefore, the identity in $G^\natural(G^\natural)$ corresponds to the $\natural$-extension
%$$([L^\vee_{G^\natural} \to P_{G^\natural \times G^\vee}], \nabla_2) =  (0, \sigma^* \alpha_{T^\vee}) + ([L^\vee_{G^\natural} \to P_{G^\natural \times G^\vee}], (\pi^\natural \times Id)^*\nabla_{A, 2}).$$

Doing the analogous calculations for $G^{\vee \natural}$, we conclude that
%$$(P^\natural, \nabla) =  (0, \sigma^* \alpha_{T^\vee}) + ([L^\vee_{G^\natural} \to P_{G^\natural \times G^\vee}], (\pi^\natural \times Id)^*\nabla_{A, 2}).$$
$$(P^\natural, \nabla) = (0, (\sigma^*\alpha_{T^\vee}) \circ \bar j^\vee + (\sigma^{\vee *}\alpha_{T}) \circ \bar j) + (P^\natural, (\pi^\natural \times \pi^{\vee \natural})^*\nabla_{A}),$$
which gives us the desired result.
\end{proof}

\begin{definition}
Let $\eta: G(K) \to G^\natural(K)$ and $\eta^\vee: G^\vee(K) \to G^{\vee \natural}(K)$ be a pair of splittings of the exact sequences of Lie groups
\begin{gather}
0 \to \omega_{G^\vee}(K) \xrightarrow{\zeta} G^\natural(K) \xrightarrow{\theta} G(K) \to 0, \label{UVE1} \\
0 \to \omega_G(K) \xrightarrow{\zeta^\vee} G^{\vee \natural}(K) \xrightarrow{\theta^\vee} G^\vee(K) \to 0. \label{UVE2}
\end{gather}
We say that $(\eta, \eta^\vee)$, or also that $(\Lie \eta, \Lie \eta^\vee)$, are \emph{dual} with respect to Deligne's pairing $(\us, \us)^{Del}_M$ if
\[(\us, \us)^{Del}_M \circ (\Lie \eta, \Lie \eta^\vee) = 0.\]
We define \emph{dual} splittings with respect to $(\us, \us)^{Del}_A$ and $(\us, \us)_T$ analogously.
\end{definition}

For the proof of Lemma \ref{lem:etas} we will need the following result, which is a slight generalization of Lemma 3.1.1 in \cite[p. 641]{CO91}.
\begin{lemma} \label{lem:spl}
Let
\[0 \to V \to X \to Y \to 0\]
be an exact sequence of algebraic $K$-groups with $V$ a vector group. There is a bijection between splittings of the exact sequence
\begin{equation} \label{lem:spl-es1}
0 \to V(K) \to X(K) \to Y(K) \to 0
\end{equation}
and splittings of the exact sequence of Lie algebras
\begin{equation} \label{lem:spl-es2}
0 \to \Lie V(K) \to \Lie X(K) \to \Lie Y(K) \to 0.
\end{equation}
%\[\begin{tikzcd}[column sep=large]
%G(K) \ar[r, "s"] \ar[d, "\lambda_X"'] & G^\natural(K) \ar[d, "\lambda_{G^\natural}"] \\
%\Lie G \ar[r] \ar[r, "r"] & \Lie G^\natural
%\rlap{\ .}
%\end{tikzcd}\]
\end{lemma}

\begin{proof}
Consider the following commutative diagram
\[\begin{tikzcd}
0 \ar[r] & V(K) \ar[d, equal] \ar[r] & X(K) \ar[r] \ar[d, "\lambda_{X}"] & Y(K) \ar[d, "\lambda_Y"] \ar[r] & 0 \\
0 \ar[r] & \Lie V(K) \ar[r] & \Lie X(K) \ar[r] \ar[r] & \Lie Y(K) \ar[r] & 0
\rlap{\ .}
\end{tikzcd}\]
If $s: X(K) \to V(K)$ is a splitting of \eqref{lem:spl-es1} then $\Lie s: \Lie X(K) \to \Lie V(K)$ is a splitting of \eqref{lem:spl-es2} that satisfies $ \Lie s \circ \lambda_X = s$. For the converse, let $r: \Lie X(K) \to \Lie V(K)$ be a splitting of \eqref{lem:spl-es2}. Then 
$$s: X(K) \xrightarrow{\lambda_X}  \Lie X(K) \xrightarrow{r} \Lie V(K) = V(K)$$ 
is a splitting of \eqref{lem:spl-es1}. Moreover, by the properties of the logarithm (see \cite[p. 5]{ZA96}), this map is such that $\Lie s = r$.
\end{proof}    

\begin{lemma} \label{lem:etas}
%There exists a splitting $\tilde \eta: G(K) \to G^\natural(K)$ of \eqref{UVE1} such that $\tilde \eta_T:= \eta|_{T(K)}: T(K) \to T^\natural(K)$ is a homomorphic section of the projection $T^\natural(K) \to T(K)$ and the homomorphism $\tilde \eta_A: A(K) \to A^\#(K)$ induced by $\tilde \eta$ is a section of $\theta_A$
Let $\eta: G(K) \to G^\natural(K)$ and $\eta^\vee: G^\vee(K) \to G^{\vee \natural}(K)$ be a pair of splittings of \eqref{UVE1} and \eqref{UVE2}, respectively. Then we can define new splittings $\tilde \eta, \tilde \eta^\vee$ such that
\begin{align*}
\Lie \tilde \eta := \Lie \eta_T \times \Lie \eta_A & : \Lie G(K) \cong \Lie T(K) \times \Lie A(K) \to \Lie T^\natural(K) \times \Lie A^\#(K) \cong \Lie G^\natural(K), \\
\Lie \tilde \eta^\vee := \Lie \eta_T^\vee \times \Lie \eta_A^\vee & : \Lie G^\vee(K) \cong \Lie T^\vee(K) \times \Lie A^\vee(K) \to \Lie T^{\vee \natural}(K) \times \Lie A^{\vee \#}(K) \cong \Lie G^{\vee \natural}(K),
\end{align*}
where $\eta_T: T(K) \to T^\natural(K), \eta_T^\vee: T^\vee(K) \to T^{\vee \natural}(K)$ are homomorphic sections of the projections
$$pr_2: T^\natural(K) \to T(K), \quad pr_2: T^{\vee \natural}(K) \to T^\vee(K),$$ 
respectively, and $\eta_A: A(K) \to A^\#(K), \eta_A^\vee: A^\vee(K) \to A^{\vee \#}(K)$ are homomorphic sections of
$$\theta_A: A^\#(K) \to A(K), \quad \theta_{A^\vee}: A^{\vee \#}(K) \to A^\vee(K),$$
respectively. Moreover, if $(\eta, \eta^\vee)$ are dual with respect to $(\us, \us)^{Del}_M$ then $(\eta_T, \eta_T^\vee)$ are dual with respect to $(\us, \us)_T$, $(\eta_A, \eta_A^\vee)$ are dual with respect to $(\us, \us)^{Del}_A$, and $(\tilde \eta, \tilde \eta^\vee)$ are dual with respect to $(\us, \us)^{Del}_M$.
\end{lemma}

\begin{proof}
Define $r_T: \Lie T(K) \to \Lie T^\natural(K)$ and $r_A: \Lie A(K) \to \Lie A^\#(K)$ such that they make the following diagram commute (see diagrams \eqref{es_la}, \eqref{es_lanat} for notation)
\[\begin{tikzcd}
\Lie T(K) \ar[r, "j"] \ar[d, dashed, "r_T"'] & \Lie G(K) \ar[d, "\Lie \eta"] & \Lie A(K) \ar[l, "\bar q"'] \ar[d, dashed, "r_A"] \\
\Lie T^\natural(K) & \Lie G^\natural(K) \ar[l, "\bar j^\natural"'] \ar[r, "q^\natural"] & \Lie A^\#(K)
\rlap{\ .}
\end{tikzcd}\]
From the definitions of $\bar j^\natural$ and $q^\natural$ we get that $r_T$ and $r_A$ are sections of 
\begin{gather*}
pr_2: \Lie T^\natural(K) \to \Lie T(K), \\
\Lie \theta_A: \Lie A^\#(K) \to \Lie A(K),
\end{gather*}
respectively. Notice that $r_T: \Lie T(K) \to \Lie T^\natural(K)$ is given by $r_T(z) = (\Lie (\sigma \circ \eta \circ \iota)(z), z)$. By Lemma \ref{lem:spl}, we can extend these homomorphisms in a canonical way to homomorphisms of Lie groups $\eta_T: T(K) \to T^\natural(K)$ and $\eta_A: A(K) \to A^\#(K)$, \textit{i.e.} satisfying $\Lie \eta_T = r_T$ and $\Lie \eta_A = r_A$, in such a way that they are sections of
\begin{gather*}
pr_2: T^\natural(K) \to T(K), \\
\theta_A: A^\#(K) \to A(K),
\end{gather*}
respectively. Notice that $\eta_T: T(K) \to T^\natural(K)$ is given by $\eta_T(t) = (\sigma \circ \eta \circ \iota(t), t)$. Let
$$\tilde r := r_T \times r_A: \Lie G(K) \cong \Lie T(K) \times \Lie A(K) \to \Lie T^\natural(K) \times \Lie A(K) \cong \Lie G^\natural(K)$$
and define $\tilde \eta: \Lie G(K) \to \Lie G^\natural$ as the morphism such that $\Lie \tilde \eta = \tilde r$. Clearly, $\tilde \eta$ is a section of $\theta$. We define $\eta^\vee_T$, $\eta^\vee_A$ and $\tilde \eta^\vee$ analogously. \\

Now suppose that $(\eta, \eta^\vee)$ are dual with respect to $(\us, \us)^{Del}_M$. We will prove that $(\eta_T, \eta_T^\vee)$ are dual splittings with respect to $(\us, \us)^{Del}_T$. By Lemma \ref{lem:Del_pairing_M}, we get the following equality for every $z \in \Lie T(K)$ and $z^\vee \in \Lie T^\vee(K)$
\begin{align*}
(\Lie \eta \circ j(z), \Lie\eta^\vee \circ j^\vee(z^\vee))^{Del}_M & = (\bar j^\natural \circ \Lie \eta \circ j(z), \bar j^{\vee \natural} \circ \Lie \eta^\vee \circ j^\vee(z^\vee))_T \\
& \quad + (q^\natural \circ \Lie \eta \circ j(z), q^{\vee \natural} \circ \Lie \eta^\vee \circ j^\vee(z^\vee))^{Del}_A . \nonumber
\end{align*}
Notice that $q^\natural \circ \Lie \eta \circ j: \Lie T(K) \to \Lie A^\#(K)$ becomes zero when composed with $\Lie \theta_A$ (see Definition \ref{def:jnat+qnat} and diagrams \eqref{notation_UVE2} and \eqref{es_Gnat}):
\begin{align*}
\Lie \theta_A \circ q^\natural \circ \Lie \eta \circ j & = \Lie(\theta_A \circ \pi^\# \circ \bar \gamma \circ \eta) \circ j \\
& = q \circ \Lie(\theta' \circ \bar \gamma \circ \eta) \circ j \\
& = q \circ \Lie(\theta \circ \eta) \circ j \\
& = 0
\end{align*}
This means that 
$$q^\natural \circ \Lie \eta \circ j(z) = (\omega, 0) \in \Lie A^\#(K)$$
is the trivial extension of $A^\vee$ by $\mathbb G_a$ endowed with a $\natural$-structure coming from an invariant differential $\omega \in \omega_{A^\vee}(K)$. Since the same is true for $q^{\vee \natural} \circ \Lie \eta^\vee \circ j^\vee(z^\vee)$ then, by \cite[Cor. 2.1.1, p. 638]{CO91}, 
$$(q^\natural \circ \Lie \eta \circ j(z), q^{\vee \natural} \circ \Lie \eta^\vee \circ j^\vee(z^\vee))^{Del}_A = 0,$$
and so
\begin{align*}
(\Lie \eta_T(z), \Lie \eta^\vee_T(z^\vee))_T & = (\bar j^\natural \circ \Lie \eta \circ j(z), \bar j^{\vee \natural} \circ \Lie \eta^\vee \circ j^\vee(z^\vee))_T \\
& = (\Lie \eta \circ j(z), \Lie\eta^\vee \circ j^\vee(z^\vee))^{Del}_M \\
& = 0 ,
\end{align*}
\textit{i.e.} $(\eta_T, \eta^\vee_T)$ are dual splittings with respect to $(\us, \us)_T$. The proof that $(\eta_A, \eta^\vee_A)$ are dual splittings with respect to $(\us, \us)^{Del}_A$ is carried out in a similar fashion. Now, to prove that $(\tilde \eta, \tilde \eta^\vee)$ are dual with respect to $(\us, \us)^{Del}_M$ consider the following commutative diagram
\begin{equation} \label{weightfil_spl}
\begin{tikzcd}[column sep=4em]
\Lie T(K) \ar[d, "\Lie \eta_T"'] & \Lie G(K) \ar[l, "\bar j"'] \ar[d, "\Lie \tilde \eta"] \ar[r, "q"] & \Lie A(K) \ar[d, "\Lie \eta_A"] \\
\Lie T^\natural(K) & \Lie G^\natural(K) \ar[l, "\bar j^\natural"'] \ar[r, "q^\natural"] & \Lie A^\#(K)
\rlap{\ ,}
\end{tikzcd}
\end{equation} 
as well as the corresponding one for $\tilde \eta^\vee$. From this and Lemma \ref{lem:Del_pairing_M} we conclude that for every $(h, h^\vee) \in \Lie G(K) \times \Lie G^\vee(K)$
\begin{align*}
(\Lie \tilde \eta (h), \Lie \tilde \eta^\vee (h^\vee))^{Del}_M & = (\Lie \tilde \eta_T \circ \bar j(h), \Lie \tilde \eta_T^\vee \circ \bar j^\vee(h^\vee))^{Del}_T \\
& \quad + (\Lie \tilde \eta_A \circ q(h), \Lie \tilde \eta_A^\vee \circ q^\vee(h^\vee))^{Del}_A \\
& = 0.
\end{align*}
%As well as homomorphic sections $\eta_T: T(K) \to T^\natural(K)$ and $\eta_T^\vee: T^\vee(K) \to T^{\vee \natural}(K)$ of the projections
%$$pr_2: T^\natural(K) \to T(K), \quad pr_2: T^{\vee \natural}(K) \to T^\vee(K),$$ 
%and splittings $\eta_A: A(K) \to A^\#(K)$ and $\eta_A^\vee: A^\vee(K) \to A^{\vee \#}(K)$ of the exact sequences
%\begin{gather*}    
%0 \to \omega_{A^\vee}(K) \to A^\#(K) \xrightarrow{\theta_A} A(K) \to 0, \\
%0 \to \omega_A(K) \to A^{\vee \#}(K) \xrightarrow{\theta_{A^\vee}} A^\vee(K) \to 0
%\end{gather*}
%such that the following diagram commutes
%\begin{equation}
%\begin{tikzcd}[column sep=4em]
%\Lie T(K) \times \Lie T^{\vee}(K) \ar[d, "\Lie \eta_T \times \Lie \eta_T^\vee"'] & \Lie G(K) \times \Lie G^{\vee}(K) \ar[l, "\bar j \times \bar j^\vee"'] \ar[d, "\Lie \tilde \eta \times \Lie \tilde \eta^\vee"] \ar[r, "q \times q^\vee"] & \Lie A(K) \times \Lie A^{\vee}(K) \ar[d, "\Lie \eta_A \times \Lie \eta_A^\vee"] \\
%\Lie T^\natural(K) \times \Lie T^{\vee \natural}(K) & \Lie G^\natural(K) \times \Lie G^{\vee \natural}(K) \ar[l, "\bar j^\natural \times \bar j^{\vee \natural}"'] \ar[r, "q^\natural \times q^{\vee \natural}"] & \Lie A^\#(K) \times \Lie A^{\vee \#}(K)
%\rlap{\ .}
%\end{tikzcd}
%\end{equation} 
\end{proof}

%\begin{lemma} \label{lambda-spl-torus}
%Let $r: \Lie T \to \Lie T^\natural$ and $r^\vee: \Lie T^\vee \to \Lie T^{\vee \natural}$ be a pair of splittings of the exact sequences
%\[0 \to \omega_{T^\vee}(K) \to \Lie T^\natural \to \Lie T \to 0 \]
%\[0 \to \omega_T(K) \to \Lie T^{\vee \natural} \to \Lie T^\vee \to 0 , \]
%respectively, such that the composition
%\[\Lie T \otimes \Lie T^\vee \xrightarrow{r \otimes r^\vee} \Lie T^\natural \otimes \Lie T^{\vee \natural} \xrightarrow{\Phi} K\]
%is zero, \textit{i.e.} $r$ and $r^\vee$ are dual. Then we have an induced $\lambda$-splitting $\psi: K^* \times T^\natural(K) \times T^{\vee \natural}(K) \to K$.
%\end{lemma}

\begin{theorem} \label{thm:lambda-spl}
Let $r: \Lie G(K) \to \Lie G^\natural(K)$ and $r^\vee: \Lie G^\vee(K) \to \Lie G^{\vee \natural}(K)$ be a pair of splittings of the exact sequences of Lie algebras
\begin{gather*}
0 \to \omega_{G^\vee}(K) \xrightarrow{\Lie \zeta} \Lie G^\natural(K) \xrightarrow{\Lie \theta} \Lie G(K) \to 0, \\
0 \to \omega_G(K) \xrightarrow{\Lie \zeta^\vee} \Lie G^{\vee \natural}(K) \xrightarrow{\Lie \theta^\vee} \Lie G^\vee(K) \to 0 ,
\end{gather*}
respectively, which are dual with respect to $(\us, \us)^{Del}_M$. Then we have an induced $\lambda$-splitting 
\[\psi: P(K) \to K,\]
where $P$ is the Poincar\'e biextension.
\end{theorem}

\begin{proof}
Let $g \in G(K)$ be a section above $a \in A(K)$. First note that, from the splitting of $\Lie G^\vee$ in \eqref{es_la}, we also obtain a splitting of $\Lie P_{g, G^\vee}$ by pullback
\begin{equation} \label{lambda-spl}
\begin{tikzcd}
& & 0 \ar[d] & 0 \ar[d] & \\
& & \Lie \mathbb G_m \ar[d] \ar[r, equal] & \Lie \mathbb G_m \ar[d] & \\
0 \ar[r] & \Lie T^\vee \ar[d, "\cong"']  \ar[r] & \Lie P_{g, G^\vee} \arrow[dr, phantom, "\lrcorner", very near start] \ar[d] \ar[r] \ar[l, dashed, bend right] & \Lie P_{a, A^\vee} \ar[d] \ar[r] \ar[l, dashed, bend right] & 0 \\
0 \ar[r] & \{g\} \times \Lie T^\vee \ar[r, "j^\vee"'] & \{g\} \times \Lie G^\vee \ar[d] \ar[r, "q^\vee"'] \ar[l, dashed, bend right, "\bar j^\vee"'] & \{a\} \times \Lie A^\vee \ar[d] \ar[r] \ar[l, dashed, bend right, "\bar q^\vee"'] & 0 \\
& & 0 & 0 & 
\end{tikzcd}
\rlap{\ .}
\end{equation}
In a similar way, we induce a splitting of $\Lie P_{G, g^\vee}$, for all $g^\vee \in G^\vee(K)$. \\

Let $\eta: G(K) \to G^\natural(K)$ and $\eta^\vee: G^\vee(K) \to G^{\vee \natural}(K)$ be the splittings of \eqref{UVE1} and \eqref{UVE2}, respectively, such that $\Lie \eta = r$ and $\Lie \eta^\vee = r^\vee$, and let $\eta_T$, $\eta_T^\vee$, $\eta_A$, $\eta_A^\vee$, $\tilde \eta$ and $\tilde \eta^\vee$ be as constructed in Lemma \ref{lem:etas}. Consider the following diagram
%\[\begin{tikzcd}
%& G(K) \ar[d, "\tilde \eta"] & & \\
%\omega_{T^\vee}(K) & G^\natural(K) \ar[l, "\sigma"'] \ar[r, "\bar \gamma"] \ar[rr, bend left=20, "\pi^\natural"] & G^\#(K) \ar[r, "\pi^\#"] & A^\#(K)
%\rlap{\ .}
%\end{tikzcd}\] 
\begin{equation} \label{Gnat_eta}
\begin{tikzcd}
& G(K) \ar[d, "\tilde \eta"] \ar[r, "\pi"] & A(K) \ar[d, "\eta_A"'] \\
\omega_{T^\vee}(K) & G^\natural(K) \ar[l, "\sigma"'] \ar[r, "\pi^\natural"] & A^\#(K)
\rlap{\ .}
\end{tikzcd}
\end{equation}

Denote by $s_{g}^1: \Lie T^\vee \to K$ the morphism of Lie algebras corresponding to the invariant differential $\sigma \circ \tilde \eta(g) \in \omega_{T^\vee}(K)$. By \cite[Thm. 0.3.1, p. 633]{CO91} (see also Lemma \ref{lem:UVE} (iii)) we have that $\pi^\natural \circ \tilde \eta(g) \in A^\#(K)$ is represented by the $\mathbb G_m$-extension $P_{a, A^\vee}$ of $A^\vee$ equipped with a normal invariant differential, which corresponds to a morphism $s_{g}^2: \Lie P_{a, A^\vee} \to K$. We define
%\cite[Thm. 0.3.1]{CO91}
\begin{align*}
s_g: \Lie P_{g, G^\vee} \cong \Lie T^\vee \times \Lie P_{a, A^\vee} & \to K \\
z = (z^1, z^2) & \mapsto s_g^1(z^1) + s_g^2(z^2) .
\end{align*}
This is a rigidification of $P_{g, G^\vee}$, considered as an extension of $G^\vee$ by $\mathbb G_m$. For every $g^\vee \in G^\vee(K)$, we let $a^\vee := \pi^\vee(g^\vee)$, and define the rigidification $s_{g^\vee}: \Lie P_{G, g^\vee} \to K$ of $P_{G, g^\vee}$ analogously as
\begin{align*}
s_{g^\vee}: \Lie P_{G, g^\vee} \cong \Lie T \times \Lie P_{A, a^\vee} & \to K \\
z = (z^1, z^2) & \mapsto s_{g^\vee}^1(z^1) + s_{g^\vee}^2(z^2) ,
\end{align*}
where $s_{g^\vee}^1: \Lie T \to K$ is the morphism corresponding to the invariant differential $\sigma^\vee \circ \tilde \eta^\vee(g^\vee) \in \omega_{T}(K)$, and $s_{g^\vee}^2: \Lie P_{A, a^\vee} \to K$ is the morphism corresponding to the normal invariant differential on $P_{A, a^\vee}$ associated to $\pi^{\vee \natural} \circ \tilde \eta^\vee(g^\vee) \in A^{^\vee \#}(K)$. \\
    
Let $y \in P(K)$ lie above $(g, g^\vee) \in G(K) \times G^\vee(K)$. We define maps $\psi_1, \psi_2: P(K) \to K$ as follows
$$\psi_1(y) = s_g \circ \lambda_{P_{g, G^\vee}}(y), \quad \psi_2(y) = s_{g^\vee} \circ \lambda_{P_{G, g^\vee}}(y).$$
\noindent\begin{minipage}{0.55\linewidth}
\begin{equation} \label{diag_def_rho-spl1}
\begin{tikzcd}
K^* \ar[d, hook] \ar[r, "\lambda"] & K \ar[d, hook] \\
P_{g, G^\vee}(K) \ar[r, "\lambda_{P_{g, G^\vee}}"] \ar[d] & \Lie P_{g, G^\vee}(K) \ar[d] \ar[u, dashed, bend right, "s_g"'] \\
\{g\} \times G^\vee(K) \ar[r, "\lambda_{G^\vee}"] & \Lie G^\vee(K)
\end{tikzcd}
\end{equation}
\end{minipage}%
\begin{minipage}{0.5\linewidth}
%\begin{equation} \label{diag_def_rho-spl2}
\[\begin{tikzcd}
K^* \ar[d, hook] \ar[r, "\lambda"] & K \ar[d, hook] \\
P_{G, g^\vee}(K) \ar[r, "\lambda_{P_{G, g^\vee}}"] \ar[d] & \Lie P_{G, g^\vee}(K) \ar[d] \ar[u, dashed, bend right, "s_{g^\vee}"'] \\
G(K) \times \{g^\vee\} \ar[r, "\lambda_{G}"] & \Lie G(K)
\end{tikzcd}\]
%\end{equation}
\end{minipage}

\begin{claim}
$\psi_1 = \psi_2$.
\end{claim}

\begin{proof}
Denote 
\begin{align*}
(z_g^1, z_g^2) & := \lambda_{P_{g, G^\vee}}(y) \in \Lie P_{g, G^\vee} \cong \Lie T^\vee \times \Lie P_{a, A^\vee}, \\
(z_{g^\vee}^1, z_{g^\vee}^2) & := \lambda_{P_{G, g^\vee}}(y) \in \Lie P_{G, g^\vee} \cong \Lie T \times \Lie P_{A, a^\vee}.
\end{align*}
To prove the claim it suffices to show that 
\begin{align*}
s_{g}^1(z_g^1) & = s_{g^\vee}^1(z_{g^\vee}^1), \\
s_{g}^2(z_g^2) & = s_{g^\vee}^2(z_{g^\vee}^2).
\end{align*}
\begin{enumerate}[(i)]
\item $s_{g}^1(z_g^1) = s_{g^\vee}^1(z_{g^\vee}^1)$: From the commutativity of diagram \eqref{lambda-spl} and the analogous one for $P_{G, g^\vee}$ we get that
$$z^1_{g^\vee} = \bar j \circ \lambda_G(g) \in \Lie T (K), \quad  z^1_g = \bar j^\vee \circ \lambda_{G^\vee}(g^\vee) \in \Lie T^\vee(K).$$
Therefore, we have 
\begin{align*}
\Lie \eta_T(z^1_{g^\vee}) & = \Lie \eta_T \circ \bar j \circ \lambda_G(g) \\
& = \bar j^\natural \circ \Lie \tilde \eta \circ \lambda_G(g) \\
& = (\sigma \circ \Lie \tilde \eta \circ \lambda_G(g), \bar j \circ \Lie \theta' \circ \Lie \gamma' \circ \Lie \tilde \eta \circ \lambda_G(g)) \\
& = (\sigma \circ \tilde \eta(g), \bar j \circ \Lie \theta \circ \Lie \tilde \eta \circ \lambda_G(g)) \\
& = (\sigma \circ \tilde \eta(g), \bar j \circ \lambda_G(g)) \in \omega_{T^\vee}(K) \times \Lie T(K) = \Lie T^\natural(K),
\end{align*}
where the second equality comes from the commutativity of diagram \eqref{weightfil_spl} in the proof of Lemma \ref{lem:etas}, the third one from the definition of $\bar j^\natural$ (see diagram \eqref{es_lanat}), the fourth one from the fact that $\theta' \circ \gamma' = Id$, and the last one from the fact that $\theta \circ \tilde \eta = Id$. Similarly, 
$$\Lie \eta^\vee_T(z^1_g) = (\sigma^\vee \circ \tilde \eta^\vee(g^\vee), \bar j^\vee \circ \lambda_{G^\vee}(g^\vee)) \in \omega_{T}(K) \times \Lie T^\vee(K) = \Lie T^{\vee \natural}(K).$$ 
By Lemma \ref{lem:Del_pairing_T}, we have
$$(\Lie \eta_T(z^1_{g^\vee}), \Lie \eta^\vee_T(z^1_g))_T = s^1_g(z^1_g) - s^1_{g^\vee}(z^1_{g^\vee}).$$
Since $(\eta_T, \eta^\vee_T)$ are dual, we get the desired equality.
    
\item $s_{g}^2(z_g^2) = s_{g^\vee}^2(z_{g^\vee}^2)$: Let $y_A \in P_A(K)$ be the image of $y$. Then, by functoriality of the logarithm, we get 
$$z^2_g = \lambda_{P_{a, A^\vee}}(y_A), \quad z^2_{g^\vee} = \lambda_{P_{A, a^\vee}}(y_A).$$ 
Notice that, because of the commutativity of diagram \eqref{Gnat_eta}, we have 
\begin{align*}
\eta_A(a) & = \eta_A \circ \pi (g) \\
& = \pi^\natural \circ \tilde \eta(g) \in A^\#(K).
\end{align*}
Similarly,
$$\eta^\vee_A(a^\vee) = \pi^{\vee \natural} \circ \tilde \eta^\vee(g) \in A^{\vee \#}(K).$$ 
Hence, if we denote by $s_a$ the rigidification of $P_{a, A^\vee}$ determined by $\eta_A(a)$ and by $s_{a^\vee}$ the rigidification of $P_{A, a^\vee}$ determined by $\eta^\vee_A(a^\vee)$ then $s_a = s_{g}^2$ and $s_{a^\vee} = s_{g^\vee}^2$. Since $(\eta_A, \eta^\vee_A)$ are dual, the $\lambda$-splittings of $P_A(K)$ obtained from $\eta_A$ and $\eta^\vee_A$ coincide (see Proposition 3.1.2, Corollary 3.1.3 and Proposition 3.2.1 in \cite[p. 642--643]{CO91}). This implies that
\begin{align*}
s_{g}^2(z_g^2) & = s_a \circ \lambda_{P_{a, A^\vee}}(y_A) \\
& = s_{a^\vee} \circ \lambda_{P_{A, a^\vee}}(y_A) \\
& = s_{g^\vee}^2(z_{g^\vee}^2) .
\end{align*}
\end{enumerate}
\end{proof}

Therefore, we can define 
$$\psi := \psi_1 = \psi_2.$$
It only remains to check that $\psi$ is indeed a $\lambda$-splitting. Using the definition of $\psi_1$ we get that for all $c \in K^*$ and $y \in P(K)$ lying above $(g, g^\vee) \in G(K) \times G^\vee(K)$
\begin{align*}
\psi(c + y) & = s_g \circ \lambda_{P_{g, G^\vee}}(c + y) \\
& = s_g \circ \lambda_{P_{g, G^\vee}}(c) + s_g \circ \lambda_{P_{g, G^\vee}}(y) \\
& = \lambda(c) + \psi(y) ,
\end{align*}
where the last equality holds because of the commutativity of diagram \eqref{diag_def_rho-spl1}. Also, for $y, y' \in P_{g, G^\vee}(K)$,
\begin{align*}
\psi(y + y') & = s_g \circ \lambda_{P_{g, G^\vee}}(y +_1 y') \\
& = s_g \circ \lambda_{P_{g, G^\vee}}(y) + s_g \circ \lambda_{P_{g, G^\vee}}(y') \\
& = \psi(y) + \psi(y').
\end{align*}
Finally, from the definition of $\psi_2$ it follows that $\psi$ is also compatible with the group structure $+_2$ of $P(K)$.
\end{proof}

%Check that I write $\lambda_G(u(x))$ and not $\lambda_G \circ u(x)$. #checkthis

\begin{theorem}
In the situation of Theorem \ref{thm:lambda-spl}, assume that $\eta$ and $\eta^\vee$ make the following diagrams commute
\[\xymatrix{ 
L(K) \ar@{=}[r] \ar[d]_{u} & L(K) \ar[d]^{u^\natural} \\
G(K) \ar[r]^{\eta} & G^\natural(K)
}
\quad
\xymatrix{
L^\vee(K) \ar@{=}[r] \ar[d]_{u^\vee} & L^\vee(K) \ar[d]^{u^{\vee \natural}} \\
G^\vee(K) \ar[r]^{\eta^\vee} & G^{\vee \natural}(K)}\]
and, moreover, that $\eta = \tilde \eta$, $\eta^\vee = \tilde \eta^\vee$, where $\tilde \eta$ and $\tilde \eta^\vee$ are the morphisms of Lemma \ref{lem:etas}. Then the $\lambda$-splitting $\psi: P(K) \to K$ constructed in Theorem \ref{thm:lambda-spl} is compatible with the $L \times L^\vee$-linearization of $P$. In particular, it induces a $\lambda$-splitting of the biextension $Q_M(K)$ of $(M(K), M^\vee(K))$ by $K^*$ in the case that $u(K)$ and $u^\vee(K)$ are injective. 
\end{theorem}

\begin{remark}
The condition $\eta \circ u = u^\natural$ says that, on $K$-sections, $(Id, \eta)$ is a splitting of the complex $M^\natural$ seen as an extension of $M$ by $\omega_{G^\vee}$; and similarly for $\eta^\vee$.
\end{remark}

\begin{proof}
We have to prove that the $\lambda$-splitting $\psi: P(K) \to K$ constructed in Theorem \ref{thm:lambda-spl} satisfies $\psi \circ \tau = 0$ and $\psi \circ \tau^\vee = 0$ on $K$-sections. \\

Let $x \in L(K)$ and denote by $\chi: T^\vee \to \mathbb G_m$ the homomorphism corresponding to it. We have the following diagram with exact rows (see \cite[\S 1.2]{AB05})
\[\begin{tikzcd}
0 \ar[r] & T^\vee \ar[r, "\iota^\vee"] \ar[d, "-\chi"'] & G^\vee \ar[r, "\pi^\vee"] \ar[d, "\tau'_x"] & A^\vee \ar[r] \ar[d, equal] & 0 \\
0 \ar[r] & \mathbb G_m \ar[r] & P_{v(x), A^\vee} \ar[r] & A^\vee \ar[r] & 0 \\
0 \ar[r] & \mathbb G_m \ar[r] \ar[u, equal] & P_{u(x), G^\vee} \ar[r] \ar[u] & G^\vee \ar[r] \ar[u] & 0
\rlap{\ ,}
\end{tikzcd}\]
where $v$ is the composition $L \xrightarrow{u} G \xrightarrow{\pi} A$. We also have the corresponding diagram of Lie algebras with exact rows and splittings induced by $\bar j^\vee$ and $\bar q^\vee$:
\begin{equation} \label{spl_P}
\begin{tikzcd}
0 \ar[r] & \Lie T^\vee \ar[r, "j^\vee"] \ar[d, "-\Lie \chi"'] & \Lie G^\vee \ar[r, "q^\vee"] \ar[d] \ar[l, dashed, bend right, "\bar j^\vee"'] & \Lie A^\vee \ar[r] \ar[d, equal] \ar[l, dashed, bend right, "\bar q^\vee"'] & 0 \\
0 \ar[r] & \Lie \mathbb G_m \ar[r] & \Lie P_{v(x), A^\vee} \ar[r] \ar[l, dashed, bend right, "\xi"'] & \Lie A^\vee \ar[r] \ar[l, dashed, bend right] & 0 \\
0 \ar[r] & \Lie \mathbb G_m \ar[r] \ar[u, equal] & \Lie P_{u(x), G^\vee} \ar[r] \ar[u] \ar[l, dashed, bend right] & \Lie G^\vee \ar[r] \ar[u] \ar[l, dashed, bend right] & 0
\rlap{\ .}
\end{tikzcd}
\end{equation}

By Lemma \ref{lem:UVE} (i), $u^\natural(x) \in G^\natural(K)$ corresponds to the extension $[L^\vee \to P_{u(x), G^\vee}]$ of $M^\vee$ by $\mathbb G_m$ endowed with a $\natural$-structure. We know that the invariant differential $\sigma \circ u^\natural(x) \in \omega_{T^\vee}(K)$ is the one associated to the homomorphism $\Lie \chi \in \Hom_{\mathcal O_K}(\Lie T^\vee, \mathbb G_a)$, by Lemma \ref{lem:UVE} (iv).
%\cite[Prop. 3.8, p. 1643]{BE09}
%\begin{align*}
%ev: \Hom_K(T^\vee, \mathbb G_m) & \rightarrow \Hom_{\mathcal O_K}(\sHom_K(L, \mathbb G_a), \mathcal O_K) = \Hom_{\mathcal O_K}(\Lie T^\vee, \mathbb G_a) \\
%\chi & \mapsto (ev(x): \sHom_K(L, \mathbb G_a) \to \mathcal O_K) = \Lie \chi .
%\end{align*}
On the other hand, $\pi^\natural \circ u^\natural(x) \in A^\#(K)$ is the extension $P_{v(x), A^\vee}$ of $A^\vee$ by $\mathbb G_m$ endowed with the normal invariant differential associated to $\xi: \Lie P_{v(x), A^\vee} \to \Lie \mathbb G_m$. From our hypothesis that $\eta \circ u = u^\natural$, it follows that 
$$s_{u(x)}^1 = \Lie \chi: \Lie T^\vee \to \Lie \mathbb G_m,$$
since this is the morphism induced by $\sigma \circ \eta(u(x)) = \sigma \circ u^{\natural}(x)$, and 
$$s_{u(x)}^2 = \xi: \Lie P_{v(x), A^\vee} \to \Lie \mathbb G_m,$$ 
since this is the morphism induced by $\pi^\natural \circ \eta(u(x)) = \pi^\natural \circ u^\natural(x)$. \\

Let $g^\vee \in G^\vee(K)$. By setting $g = u(x)$, the middle row in diagram \eqref{lambda-spl} provides us with a decomposition $\Lie P_{u(x), G^\vee} \cong \Lie T^\vee \times \Lie P_{v(x), A^\vee}$ identifying 
$$\lambda_{P_{u(x), G^\vee}}(\tau(x, g^\vee)) = (\bar j^\vee \circ \lambda_{G^\vee}(g^\vee), \lambda_{P_{v(x), A^\vee}} \circ \tau'_x(g^\vee)).$$
Furthermore, the middle row of diagram \eqref{spl_P}, allows us to identify $\Lie P_{v(x), A^\vee}$ with $\Lie \mathbb G_m \times \Lie A^\vee$; under this isomorphism, $\lambda_{P_{v(x), A^\vee}} \circ \tau'_x(g^\vee)$ corresponds to
$$\lambda_{P_{v(x), A^\vee}} \circ \tau'_x(g^\vee) = (-\Lie \chi \circ \bar j^\vee \circ \lambda_{G^\vee}(g^\vee), \lambda_{A^\vee}(a^\vee)),$$ 
where $a^\vee \in A^\vee$ is the image of $g^\vee \in G^\vee$ under the canonical projection. Therefore, by \eqref{diag_def_rho-spl1}, we get that $\psi \circ \tau(x, g^\vee)$ equals
\begin{align*}
\psi \circ \tau(x, g^\vee) & = s_{u(x)} \circ \lambda_{P_{u(x), G^\vee}}(\tau(x, g^\vee)) \\
& = s^1_{u(x)}(\bar j^\vee \circ \lambda_{G^\vee}(g^\vee)) + s^2_{u(x)}(\lambda_{P_{v(x), A^\vee}} \circ \tau'_x(g^\vee)) \\
& = \Lie \chi \circ \bar j^\vee \circ \lambda_{G^\vee}(g^\vee) + \xi(-\Lie \chi \circ \bar j^\vee \circ \lambda_{G^\vee}(g^\vee), \lambda_{A^\vee}(a^\vee)) \\
& = \Lie \chi \circ \bar j^\vee \circ \lambda_{G^\vee}(g^\vee) - \Lie \chi \circ \bar j^\vee \circ \lambda_{G^\vee}(g^\vee) \\
& = 0 .
\end{align*}

The proof of the equality $\psi \circ \tau^\vee(g, x^\vee) = 0$ is carried out in a similar way.
\end{proof}

\begin{corollary}
Let $\rho: K^* \to \mathbb Q_p$ be a ramified homomorphism and consider $r: \Lie G(K) \to \Lie G^\natural(K)$ and $r^\vee: \Lie G^\vee(K) \to \Lie G^{\vee \natural}(K)$ a pair of splittings of the exact sequences of Lie algebras
\[0 \to \omega_{G^\vee}(K) \xrightarrow{\Lie \zeta} \Lie G^\natural(K) \xrightarrow{\Lie \theta} \Lie G(K) \to 0,\]
\[0 \to \omega_G(K) \xrightarrow{\Lie \zeta^\vee} \Lie G^{\vee \natural}(K) \xrightarrow{\Lie \theta^\vee} \Lie G^\vee(K) \to 0,\]
respectively, which are dual with respect to $(\us, \us)^{Del}_M$. Then:
\begin{enumerate}[(i)]
\item There is a $\rho$-splitting $\psi: P(K) \to \mathbb Q_p$.
\item Let $\eta: G(K) \to G^\natural(K)$ and $\eta^\vee: G^\vee(K) \to G^{\vee \natural}(K)$ be the splittings of \eqref{UVE1} and \eqref{UVE2} such that $\Lie \eta = r$ and $\Lie \eta^\vee = r^\vee$. If the following diagrams commute
\[\xymatrix{ 
L(K) \ar@{=}[r] \ar[d]_{u} & L(K) \ar[d]^{u^\natural} \\
G(K) \ar[r]^{\eta} & G^\natural(K)
}
\quad
\xymatrix{
L^\vee(K) \ar@{=}[r] \ar[d]_{u^\vee} & L^\vee(K) \ar[d]^{u^{\vee \natural}} \\
G^\vee(K) \ar[r]^{\eta^\vee} & G^{\vee \natural}(K)
}\]
and $\eta = \tilde \eta$, $\eta^\vee = \tilde \eta^\vee$, where $\tilde \eta$ and $\tilde \eta^\vee$ are the morphisms of Lemma \ref{lem:etas}, then the $\rho$-splitting $\psi: P(K) \to \mathbb Q_p$ of (i) is compatible with the $L \times L^\vee$-linearization of $P$. In particular, if $u(K)$ and $u^\vee(K)$ are injective then $\psi$ induces a $\rho$-splitting of the biextension $Q_M(K)$ of $(M(K), M^\vee(K))$ by $K^*$.
\end{enumerate}
\end{corollary}

\begin{proof}
\begin{enumerate}[(i)]
\item 
By \cite[p. 319]{ZA90}, there exists a branch $\lambda: K^* \to K$ of the $p$-adic logarithm and a $\mathbb Q_p$-linear map $\delta: K \to \mathbb Q_p$ such that
\[\xymatrix{
K^* \ar[rr]^{\rho} \ar[dr]_{\lambda} && \mathbb Q_p \\
& K \ar[ru]_{\delta} &
\rlap{\ .} }\]
Let $\psi: P(K) \to K$ be the $\lambda$-splitting constructed as in Theorem \ref{thm:lambda-spl}. Then $\psi_{\rho} := \delta \circ \psi: P(K) \to \mathbb Q_p$ is a $\rho$-splitting of $P(K)$.
\item We have that
\[\psi_{\rho} \circ \tau = \delta \circ \psi \circ \tau = 0 ,\]
and similarly for $\tau^\vee$. Therefore, $\psi_{\rho}$ is compatible with the $L \times L^\vee$-linearization of $P$ and thus induces a $\rho$-splitting of $Q_M(K)$, in the case that $u(K)$ and $u^\vee(K)$ are injective.
\end{enumerate}
\end{proof}

%----------------------------------------------------------------------------------------------------

\section{Local pairing between zero-cycles} \label{sec:local_pairing}

In this section, we construct a pairing between disjoint zero-cycles of degree zero on a curve over a local field and its regular locus, which generalizes the local pairing defined in \cite[p. 212]{MT83} in the case of an elliptic curve (see also \cite{CG89}). \\

Let $K$ be a finite extension of $\mathbb Q_p$ and $C$ a semi-normal irreducible curve over $K$. Consider the following commutative diagram
\[\begin{tikzcd}
C'  \ar[d, twoheadrightarrow, "\pi"'] \ar[r, hook, "j'"] & \bar C' \ar[d, twoheadrightarrow, "\bar \pi"] \\
C \ar[r, hook, "j"] & \bar C
\rlap{\ ,}
\end{tikzcd}\]
where $C'$ is the normalization of $C$, $\bar C'$ is a smooth compactification of $C'$, and $\bar C$ (resp. $C$) is the curve obtained from $\bar C'$ (resp. $C'$) by contracting each of the finite sets $\pi^{-1}(x)$, for $x \in C$. Let $S$ be the set of singular points of $C$, $S' := \pi^{-1}(S)$, and $F := \bar C' - C' = \bar C - C$. We recall from Section \ref{sec:alb_pic} the homological Picard 1-motive of $C$ and the cohomological Albanese 1-motive of $C$:
%Notice that $S$ is closed in $\bar C$, $S'$ is closed in $\bar C'$ and $\bar \pi$ induces a bijection $F \cong F$. 
\[\Pic^-(C) = [u: \Div^0_{S'/S}(\bar C', F) \to \Pic^0(\bar C', F)],\]
\[\Alb^+(C) = \Pic^-(C)^\vee = [u^\vee: \Div^0_{F}(\bar C) \to \Pic^0(\bar C)] .\]
Denote by $\bar C_{\reg}$ the set of smooth points of $\bar C$ and let $a^+_{x}: \bar C_\reg \to \Pic^0(\bar C)$ be the Albanese mapping, which depends on a base point $x \in \bar C_{\reg}$ (see \cite[p. 50]{BS01}). Extending by linearity, one obtains a mapping $a^+_{\bar C}: Z_0(\bar C_\reg)_0 \to \Pic^0(\bar C)$ on the group of zero-cycles of degree zero on $\bar C_\reg$; notice that it does not depend on any base point. As usual, we denote by $P$ the Poincar\'e biextension of $(\Pic^-(C), \Alb^+(C))$ by $\mathbb G_m$. We consider a homomorphism $\rho: K^* \to \mathbb Q_p$ and a $\rho$-splitting $\psi: P(K) \to \mathbb Q_p$ which is compatible with the $\Div^0_{S'/S}(\bar C', F) \times \Div^0_{F}(\bar C)$-linearization of $P$. Our aim is to construct a pairing
\[ [\us, \us]_C: (Z_0(C)_0 \times Z_0(C_\reg)_0)' \to \mathbb Q_p ,\]
where $(Z_0(C)_0 \times Z_0(C_\reg)_0)'$ denotes the subset of $Z_0(C)_0 \times Z_0(C_\reg)_0$ consisting of pairs of cycles with disjoint support. \\

First, we define a pairing
\[[\us, \us]'_C: (\Div^0(\bar C', F) \times Z_0(\bar C_\reg)_0)' \to \mathbb Q_p \]
on the set of all pairs $(D, z)$, with $D$ a divisor on $\bar C'$ algebraically equivalent to 0 whose support is contained in $\bar C' \backslash F$, and $z$ a zero-cycle of degree zero on $\bar C_\reg$, satisfying that $\supp D \cap \supp z = \emptyset$. Notice that a divisor $D \in \Div^0(\bar C', F) \subset \Div^0(\bar C')$ corresponds to a line bundle $L(D)$ over $\bar C'$ together with a rational section $s_D: \bar C' \dashrightarrow L(D)$ which is defined on the open subset $\bar C' \setminus \supp D \subset \bar C'$; in particular, $s_D$ is defined on $F$, since $\supp D \cap F = \emptyset$. Moreover, the pullback along $a_{x}^+$ of $P_{[D]}$, the fiber of the Poincar\'e bundle $P$ over $[D] \in \Pic^0(\bar C', F)$, is the restriction of $L(D)$ to $\bar C_{\reg}$, and so $a_{x}^+$ induces a map $a_{x, D}^+: L(D)|_{\bar C_{\reg}} \to P_{[D]}$ by pullback:
\[\begin{tikzcd}
L(D)|_{\bar C_{\reg}} \ar[r, "a^+_{x, D}"] \ar[d] \ar[dr, phantom, "\lrcorner", very near start] & P_{[D]} \ar[d] \\
\bar C_{\reg} \ar[r, "a^+_{x}"] \ar[u, dashed, bend left, "s_D|_{\bar C_{\reg}}"] & \{[D]\} \times \Pic^0(\bar C)
\rlap{\ .}
\end{tikzcd}\]  
Therefore, we can define
\[ [D, \sum n_j x_j]'_C := \sum n_j \psi \circ a_{x, D}^+ \circ s_D(x_j) ,\]
where $\sum n_j x_j \in Z_0(\bar C_\reg)_0$ is a zero-cycle whose support is disjoint from $\supp D$. Notice that since $\sum n_j x_j$ is a zero-cycle then $[D, \sum n_j x_j]'_C$ no longer depends on the base point $x$. \\

When $D \in \Div^0_{S'/S}(\bar C', F) \subset \Div^0(\bar C', F)$ we have that $a_{x, D}^+ \circ s_D = \tau \circ a_{x}^+$:
\[\begin{tikzcd}
L(D)|_{\bar C_{\reg}} \ar[r, "a_{x, D}^+"] \ar[d] \ar[dr, phantom, "\lrcorner", very near start] & P_{u(D)} \ar[d] \\
\bar C_{\reg} \ar[r, "a_{x}^+"] \ar[u, dashed, bend left, "s_D|_{\bar C_{\reg}}"] & \{u(D)\} \times \Pic^0(\bar C) \ar[u, dashed, bend right, "\tau |_{\{D\} \times \Pic^0(\bar C)}"']
\rlap{\ .}
\end{tikzcd}\]  
This implies that $[D, \sum n_j x_j]'_C = 0$ for all $D \in \Div^0_{S'/S}(\bar C', F)$. Notice that, since every closed point in $C'$ is also closed in $\bar C'$, then $Z_0(C')_0 = \Div^0(\bar C', F)$. Moreover, since $\bar C'$ is irreducible, $\Div^0_{S'/S}(\bar C', F) \subset \Div^0(\bar C', F)$ is the free abelian subgroup generated by cycles of the form $x_0 - x_1$, where $\pi(x_0) = \pi(x_1)$; denote this group by $Z_0(S'/S)_0$. 
%Reference: \cite[lemma 4.3]{LI02} 
Recalling that the pushforward of cycles along $\pi$ preserves the degree, we obtain the following exact sequence
\[0 \to Z_0(S'/S)_0 \to Z_0(C')_0 \xrightarrow{\pi_*} Z_0(C)_0 \to 0. \]
Therefore, $[\us, \us]'$ is a pairing on $(Z_0(C')_0 \times Z_0(\bar C_\reg)_0)'$ which is zero when restricted to $(Z_0(S'/S)_0 \times Z_0(\bar C_\reg)_0)'$, yielding a pairing
\[ [\us, \us]''_C: (Z_0(C)_0 \times Z_0(\bar C_\reg)_0)' \to \mathbb Q_p .\]
By restricting to $Z_0(C_\reg)_0 \subset Z_0(\bar C_\reg)_0$ we get the desired pairing
\[ [\us, \us]_C: (Z_0(C)_0 \times Z_0(C_\reg)_0)' \to \mathbb Q_p .\]

We make the remark that since $\bar C'$ is irreducible then $\Div_{F}^0(\bar C) = Z_0(F)_0$, and so the restriction of $a_{\bar C}^+$ to $Z_0(F)_0$ equals $u^\vee$:
\[\begin{tikzcd}
Z_0(F)_0 \ar[r, equal] \ar[d] & \Div^0_{F}(\bar C) \ar[d, "u^\vee"] \\
Z_0(\bar C_{\reg})_0 \ar[r, "a^+_{\bar C}"] & \Pic^0(\bar C)
\rlap{\ .}
\end{tikzcd}\]
Therefore, $[D, z]'_C = \psi \circ \tau^\vee(z) = 0$ for all $z \in Z_0(F)_0$:
\[\begin{tikzcd}
& & K^* \ar[d, hook] \ar[dr, bend left, "\rho"] & \\
& & P_{[D]}(K) \ar[r, "\psi"] \ar[d] & \mathbb Q_p \\
Z_0(F)_0 \ar[r, equal] & \{[D]\} \times \Div_{F}^0(\bar C) \ar[r, "u^\vee"] \ar[ur, bend left, "\tau^\vee|_{\{[D]\} \times \Div_{F}^0(\bar C)}"] & \{[D]\} \times \Pic^0(\bar C) &
\rlap{\ .}
\end{tikzcd}\]

%----------------------------------------------------------------------------------------------------

\section{Global pairing on rational points} \label{sec:global_pairing}

We define a global pairing between the rational points of a 1-motive over a global field and its dual. The construction, which is given in Proposition \ref{pro:global_pairingM}, generalizes the global pairing defined in \cite[Lemma 3.1, p. 214]{MT83} in the case of abelian varieties (see also \cite[p. 337]{ZA90}). \\

Let $F$ be a number field endowed with a set of places $\mathcal V$ which are either archimedean or discrete, and such that, for each $c \in F^*$, we have $|c|_v = 1$ for almost all $v \in \mathcal V$. For each place $v$, let $F_v$ denote the completion of $F$ with respect to $v$; for $v$ discrete denote by $\mathcal O_{F_v}$ the ring of integers of $F_v$ and let $\pi_v$ be a uniformizer of $\mathcal O_{F_v}$ such that $\pi_v \in F$. Consider a family $\rho = (\rho_v)_{v \in \mathcal V}$ of homomorphisms 
$$\rho_v: F_v^* \to \mathbb Q_p$$ 
such that $\rho_v(\mathcal O_{F_v}^*) = 0$ for almost all discrete places $v$, and such that the ``sum formula'' $\sum_v \rho_v(c) = 0$ holds for all $c \in F^*$. \\

Let $M_F = [L_F \xrightarrow{u_F} G_F]$ be a 1-motive over $F$, where $G_F$ is an extension of $A_F$ by $T_F$. For each place $v$, denote $M_{F_v} = [L_{F_v} \xrightarrow{u_{F_v}} G_{F_v}]$ its base change to $F_v$, so that $G_{F_v}$ is an extension of $A_{F_v}$ by $T_{F_v}$. Denote by $P_F$ the Poincar\'e biextension of $(M_F, M^\vee_F)$ and by $P_{F_v}$ its base change to $F_v$, which coincides with the Poincar\'e biextension of $(M_{F_v}, M^\vee_{F_v})$. Moreover, denote 
$$\tau_{F_v}: L_{F_v} \times G_{F_v}^\vee \to P_{F_v}, \ \tau^\vee_{F_v}: G_{F_v} \times L_{F_v}^\vee \to P_{F_v}$$
the trivializations associated to the 1-motive $M_{F_v}$ and its dual. \\

Observe that $M_{F_v}$ has good reduction over $\mathcal O_{F_v}$ for almost all discrete places $v$ (see \cite[Lemma 3.3, p. 309]{ABB17}). This means that there exists an $\mathcal O_{F_v}$-1-motive $M_{\mathcal O_{F_v}} = [L_{\mathcal O_{F_v}} \xrightarrow{u_{\mathcal O_{F_v}}} G_{\mathcal O_{F_v}}]$, with $G_{\mathcal O_{F_v}}$ an extension of an abelian scheme $A_{\mathcal O_{F_v}}$ by a torus $T_{\mathcal O_{F_v}}$, whose generic fiber is $M_{F_v}$. Moreover, the Poincar\'e biextension $P_{\mathcal O_{F_v}}$ of $(M_{\mathcal O_{F_v}}, M^\vee_{\mathcal O_{F_v}})$ has generic fiber equal to $P_{F_v}$ and its trivializations
$$\tau_{\mathcal O_{F_v}}: L_{\mathcal O_{F_v}} \times G_{\mathcal O_{F_v}}^\vee \to P_{\mathcal O_{F_v}}, \ \tau^\vee_{\mathcal O_{F_v}}: G_{\mathcal O_{F_v}} \times L_{\mathcal O_{F_v}}^\vee \to P_{\mathcal O_{F_v}}$$
extend $\tau_{F_v}$ and $\tau^\vee_{F_v}$, respectively. \\

Finally, for every $v$ consider a $\rho_v$-splitting $\psi_{v}: P_{F_v}(F_v) \to \mathbb Q_p$ of $P_{F_v}(F_v)$ and assume that, for almost all discrete places $v$ for which $M_{F_v}$ has good reduction, $\psi_v(P_{\mathcal O_{F_v}}(\mathcal O_{F_v})) = 0$. We have the following

\begin{proposition} \label{pro:global_pairingG}
There is a pairing
\[\langle \us,\us \rangle: G_F(F) \times G^\vee_F(F) \to \mathbb Q_p \]
such that if $y \in P_F(F)$ lies above $(g, g^\vee) \in G_F(F) \times G^\vee_F(F)$ then
\begin{equation} \label{global_pairingG}
\langle g, g^\vee \rangle = \sum_v \psi_v(y).
\end{equation}
\end{proposition}

\begin{proof}
To prove that the right hand side of \eqref{global_pairingG} is a finite sum, we use the fact that the 1-motive $M_F$ has good reduction over $\mathcal O_{F}\left[1/N\right]$ for $N$ sufficiently divisible (see \cite[Lemma 3.3, p. 309]{ABB17}). This means that $M_F$ extends to a 1-motive $M_{\mathcal O_{F}\left[1/N\right]} = [L_{\mathcal O_{F}\left[1/N\right]} \to G_{\mathcal O_{F}\left[1/N\right]}]$ over $\mathcal O_{F}\left[1/N\right]$, and similarly for $M_F^\vee$. Moreover, the Poincar\'e biextension $P_F$ extends as well to a biextension $P_{\mathcal O_{F}\left[1/N\right]}$ over $\mathcal O_{F}\left[1/N\right]$. We then have a tower of biextensions as follows:
\[\begin{tikzcd}
\mathcal O_{F}\left[1/N\right]^* \ar[r, hookrightarrow] \ar[d, hookrightarrow] & F^* \ar[d, hookrightarrow] \\
P_{\mathcal O_{F}\left[1/n\right]}(\mathcal O_{F}\left[1/N\right]) \ar[r, hookrightarrow] \ar[d, two heads] & P_F(F) \ar[d, two heads] \\
G_{\mathcal O_{F}\left[1/N\right]} \times G^\vee_{\mathcal O_{F}\left[1/N\right]} \ar[r, equal] & G_F(F) \times G^\vee_F(F)
\rlap{\ .}
\end{tikzcd}\]
Therefore, we can always choose $y \in P_{\mathcal O_{F}\left[1/n\right]}(\mathcal O_{F}\left[1/N\right])$ lying above a pair of rational points $(g, g^\vee) \in G_F(F) \times G^\vee_F(F)$. By doing so, we ensure that $y \in P_{\mathcal O_{F_v}(\mathcal O_{F_v})}$ for almost all $v$, and thus $\psi_v(y) = 0$ for almost all $v$. \\

Observe that if $y \in P_F(F)$ lies above $(g, g^\vee)$ then any other element lying above $(g, g^\vee)$ is of the form $c + y$, for $c \in F^*$. From the sum formula we obtain the equalities
\begin{align*}
\sum_v \psi_v(c + y) & =  \sum_v \rho_v(c) + \sum_v \psi_v(y) \\
& =  \sum_v \psi_v(y),
\end{align*}
which proves that the right hand side of \eqref{pro:global_pairingG} indeed defines a map $G_F(F) \times G^\vee_F(F) \to \mathbb Q_p$. It remains to check that it is bilinear. Let $y_1, y_2 \in P_F(F)$ mapping to $(g_1, g^\vee), (g_2, g^\vee) \in G_F(F) \times G^\vee_F(F)$, respectively. Since the $\psi_v$ are $\rho_v$-splittings, we get that
\begin{align*}
\langle g_1 + g_2, g^\vee \rangle & = \sum_v \psi_v(y_1 +_2 y_2) \\
& = \sum_v \psi_v(y_1) + \sum_v \psi_v(y_2) \\
& = \langle g_1, g^\vee \rangle + \langle g_2, g^\vee \rangle. 
\end{align*}
In a similar way we verify linearity in $G^\vee_F$.
\end{proof}

From now on we will assume that $L_F$ and $T_F$ are split. We assume, moreover, that $\psi_v$ factors through a $\rho_v$-splitting $\psi_{A, v}$ of $P_{A_{F_v}}(F_v)$:
$$\psi_v: P_{F_v}(F_v) \to P_{A_{F_v}}(F_v) \xrightarrow{\psi_A, v} \mathbb Q_p.$$
Denote $\mathcal V'$ the set of discrete places $v$ such that $M_{F_v}$ has good reduction and $\psi_v(P_{\mathcal O_{F_v}}(\mathcal O_{F_v})) = 0$. Notice that, necessarily, $\rho_v(\mathcal O_{F_v}^*) = 0$ for all $v \in \mathcal V'$. 

\begin{lemma} \label{formula_global_pairing}
For every $x^\vee \in L^\vee_F(F)$ and $g \in G_F(F)$ there exists $t \in T_F(F)$ such that
\[\sum_v \psi_v \circ \tau^\vee_{F_v}(g, x^\vee) = \sum_{v \in \mathcal V - \mathcal V'} \psi_v \circ \tau^\vee_{F_v}(t^{-1} g, x^\vee),\]
and similarly for every $x \in L_F(F)$ and $g^\vee \in G^\vee_F(F)$.
\end{lemma}

\begin{proof}
Fix $x^\vee \in L^\vee_F(F)$ and $g \in G_F(F)$. Suppose that $L_F^\vee \cong \mathbb Z_F^r$ and let $(m_1, \ldots, m_r) \in \mathbb Z_F^r$ be the element corresponding to $x^\vee$. Notice that this induces an isomorphism $T_F \cong \mathbb G_{m, F}^r$. Consider a discrete place $v$ in $\mathcal V'$. Since $G_{F_v}$ has good reduction then $A_{F_v}(F_v) = A_{\mathcal O_{F_v}}(\mathcal O_{F_v})$, which induces isomorphisms
\begin{equation} \label{G_class}
\frac{G_{F_v}(F_v)}{G_{\mathcal O_{F_v}}(\mathcal O_{F_v})} \cong \frac{T_{F_v}(F_v)}{T_{\mathcal O_{F_v}}(\mathcal O_{F_v})} \cong \mathbb Z^r .
\end{equation}
Since $M_{F_v}$ has good reduction, the following diagram commutes
\[\begin{tikzcd}[column sep=tiny]
& & 0 \ar[rr, hookrightarrow] & & \mathbb Q_p \\
& P_{\mathcal O_{F_v}}(\mathcal O_{F_v}) \ar[d] \ar[ur, "\psi_v |_{P_{\mathcal O_{F_v}}}"] \ar[rr, hookrightarrow] & & P_{F_v}(F_v) \ar[d] \ar[ur, "\psi_v"] & \\
& G_{\mathcal O_{F_v}}(\mathcal O_{F_v}) \times G^\vee_{\mathcal O_{F_v}}(\mathcal O_{F_v}) \ar[rr, hookrightarrow] & & G_{F_v}(F_v) \times G^\vee_{F_v}(F_v) & \\ 
G_{\mathcal O_{F_v}}(\mathcal O_{F_v}) \times L^\vee_{\mathcal O_{F_v}}(\mathcal O_{F_v}) \ar[ur, "Id \times u^\vee_{\mathcal O_{F_v}}"'] \ar[rr, hookrightarrow] \ar[uur, bend left, "\tau^\vee_{\mathcal O_{F_v}}"] & & G_{F_v}(F_v) \times L^\vee_{F_v}(F_v) \ar[ur, "Id \times u^\vee_{F_v}"'] \ar[uur, bend left, "\tau^\vee_{F_v}"] & &
\rlap{\ .}
\end{tikzcd}\]
This implies that the map $\psi_v \circ \tau^\vee_{F_v}(\us, x^\vee)$ factors through the quotient $G_{F_v}(F_v)/ G_{\mathcal O_{F_v}}(\mathcal O_{F_v})$. Thus, any $t_v \in T_{F_v}(F_v)$ whose class in $T_{F_v}(F_v)/T_{\mathcal O_{F_v}}(\mathcal O_{F_v})$ equals that of $g$ satisfies
$$\psi_v \circ \tau^\vee_{F_v}(g, x^\vee) = \psi_v \circ \tau^\vee_{F_v}(t_v, x^\vee),$$
where we identify $t_v$ with the corresponding point in $G_{F_v}(F_v)$. If the class of $g$ corresponds to $(n_1, \ldots, n_r) \in \mathbb Z^r$ under the isomorphism \eqref{G_class}, we may choose $t_v$ of the form $t_v := (\pi_v^{n_{1}}, \ldots, \pi_v^{n_{r}})$; in this way, $t_v$ belongs to $T_F(F)$ and $\psi_w \circ \tau^\vee_{F_w}(t_v, x^\vee) = 0$, for all $w \in \mathcal V'$ such that $w \neq v$.
%\[\begin{tikzcd}
%& \mathbb G_{m, F_w} \ar[r, equal] \ar[d] & \mathbb G_{m, F_w} \ar[r, equal] \ar[d] & \mathbb G_{m, F_w} \ar[d] & \\
%& \mathbb G_{m, F_w} \times T_{F_w} \ar[r] \ar[d] \ar[dr, phantom, "\lrcorner", very near start] & P_{G_{F_w}, \{x^\vee\}} \ar[r] \ar[d] \ar[dr, phantom, "\lrcorner", very near start] & P_{A_{F_w}, a^\vee} \ar[d] & \\
%0 \ar[r] & T_{F_w} \times \{x^\vee\} \ar[r] & G_{F_w} \times \{x^\vee\} \ar[r] \ar[u, "\tau_{F_w}^\vee", dashed, bend left] & A_{F_w} \times \{a^\vee\} \ar[r] & 0
%\end{tikzcd}\]
To prove this last assertion, start by considering any place $w \in \mathcal V$. We have the following commutative diagram with exact rows
\[\begin{tikzcd}
& & \mathbb G_{m, F_w} \ar[r, equal] \ar[d] & \mathbb G_{m, F_w} \ar[d] & \\
0 \ar[r] & T_{F_w} \ar[r, "i"] \ar[d, "\cong"] & P_{G_{F_w}, \{x^\vee\}} \ar[r] \ar[d] \ar[dr, phantom, "\lrcorner", very near start] & P_{A_{F_w}, a^\vee} \ar[r] \ar[d] & 0 \\
0 \ar[r] & T_{F_w} \times \{x^\vee\} \ar[r] & G_{F_w} \times \{x^\vee\} \ar[r] \ar[u, "\tau_{F_w}^\vee", dashed, bend left] & A_{F_w} \times \{a^\vee\} \ar[r] & 0
\rlap{\ ,}
\end{tikzcd}\]
where $a^\vee \in A_{F_w}^\vee(F_w)$ denotes the image of $x^\vee$ by the composition $L_{F_w}^\vee \xrightarrow{u_{F_w}} G_{F_w}^\vee \to A_{F_w}^\vee$.
The map $i$ is the one such that when composed with $P_{G_{F_w}, \{x^\vee\}} \to G_{F_w} \times \{x^\vee\}$ equals the natural injection and when composed with $P_{G_{F_w}, \{x^\vee\}} \to P_{A_{F_w}, a^\vee}$ equals zero. Let $\chi: T_F \to \mathbb G_{m, F}$ be the map corresponding to $x^\vee \in L_F^\vee$. With this notation we have 
$$\tau_{F_w}^\vee(t, x^\vee) = \chi(t) + i(t),$$ 
for all $t \in T_{F_w}$. In particular, for $w \neq v$ in $\mathcal V'$ and $t = t_v$ we get
\begin{align*}
\psi_w \circ \tau^\vee_{F_w}(t_v, x^\vee) & = \psi_w(\chi(t_v) + i(t_v)) \\
& = \rho_w(\chi(t_v)) \\
& = \rho_w(\pi_v^{\sum n_i m_i}) \\
& = (n_1m_1 + \ldots + n_rm_r)\rho_w(\pi_v) \\
& = 0,
\end{align*}
where the second equality is deduced from $\psi_w(i(t_v)) = 0$ (since $\psi_w$ is obtained from a $\rho_w$-splitting of $P_{A_{F_w}}$), and the last one from the fact that $\pi_v \in \mathcal O_{F_w}^*$. \\

Define
$$t := \prod_{v \in \mathcal V'} t_v \in T_F(F).$$
Notice that this is a finite product, since $g \in G_{\mathcal O_{F_v}}(\mathcal O_{F_v})$ for almost all $v \in \mathcal V'$. From the previous equalities, we get that $t$ satisfies
$$\psi_v \circ \tau^\vee_{F_v}(t, x^\vee) = \psi_v \circ \tau^\vee_{F_v}(g, x^\vee),$$
for every $v \in \mathcal V'$. Therefore, we obtain
\begin{align*}
\sum_v \psi_v \circ \tau^\vee_{F_v}(g, x^\vee) & = \sum_{v \in \mathcal V - \mathcal V'} \psi_v \circ \tau^\vee_{F_v}(g, x^\vee) + \sum_{v \in \mathcal V'} \psi_v \circ \tau^\vee_{F_v}(g, x^\vee) \\
& = \sum_{v \in \mathcal V - \mathcal V'} \psi_v \circ \tau^\vee_{F_v}(g, x^\vee) + \sum_{v \in \mathcal V'} \psi_v \circ \tau^\vee_{F_v}(t, x^\vee) \\
& = \sum_{v \in \mathcal V - \mathcal V'} \psi_v \circ \tau^\vee_{F_v}(g, x^\vee) - \sum_{v \in \mathcal V - \mathcal V'} \psi_v \circ \tau^\vee_{F_v}(t, x^\vee) \\
& = \sum_{v \in \mathcal V - \mathcal V'} \psi_v \circ \tau^\vee_{F_v}(t^{-1} g, x^\vee) ,
\end{align*}
where the third equality is derived from
$$\sum_{v} \psi_v \circ \tau^\vee_{F_v}(t, x^\vee) = \sum_v \rho_v(\chi(t)) = 0.$$
\end{proof}

%%% Def. trivializaciones en seccin de preliminares

\begin{proposition} \label{pro:global_pairingM}
Suppose that $u_F(K)$ and $u_F^\vee(K)$ are injective, and that the $\rho_v$-splittings $\psi_v$ are compatible with the \sloppy $L_{F_v} \times L_{F_v}^\vee$-linearization of $P_{F_v}$, for every place $v \in \mathcal V - \mathcal V'$. Then the pairing $\langle \us,\us \rangle$ of Proposition \ref{pro:global_pairingG} descends to a pairing
\[\langle \us,\us \rangle_M: M_F(F) \times M^\vee_F(F) \to \mathbb Q_p. \]
\end{proposition}

\begin{proof}
Fix $g \in G_F(F)$ and $x^\vee \in L^\vee_F(F)$, and let $t \in T_F(F)$ be the element constructed in Lemma \ref{formula_global_pairing}. We have
\[\sum_v \psi_v \circ \tau^\vee_{F_v}(g, x^\vee) = \sum_{v \in \mathcal V - \mathcal V'} \psi_v \circ \tau^\vee_{F_v}(t^{-1} g, x^\vee) = 0 .\]
Since we have the analogous equality for every $x \in L_F(F)$ and $g^\vee \in G^\vee_F(F)$, then $\langle \us, \us \rangle$ is zero on $G(F) \times \Ima(u^\vee(F))$ and $\Ima(u(F)) \times G^\vee(F)$, inducing a pairing
\[\langle \us,\us \rangle_M: M_F(F) \times M^\vee_F(F) \to \mathbb Q_p. \]
\end{proof}

%----------------------------------------------------------------------------------------------------
%----------------------------------------------------------------------------------------------------

\bibliography{Bibliography} 
\bibliographystyle{abbrv}

\vspace{2em}

%\pagestyle{plain}
%\nocite{*}
%\printbibliography

\end{document}